\documentclass[11pt]{amsart}
\usepackage{pdfsync}
\usepackage{graphicx,color,epsfig}
\usepackage{amsfonts,amsmath,amssymb,amsthm}
\usepackage{comment}
\usepackage{hyperref}
\usepackage[utf8]{inputenc}

\newtheorem{thm}{Theorem}[section]

\newtheorem{lem}[thm]{Lemma}
\newtheorem{cor}[thm]{Corollary}

\newtheorem{asm}{Assumption}

\theoremstyle{remark}
\newtheorem{rem}{Remark}[section]

\theoremstyle{definition}
\newtheorem{defn}[thm]{Definition}

\newtheorem*{cl}{Claim}

\newcommand{\ra}{\rightarrow}

\newcommand{\R}{\mathbb R}     
\newcommand{\Z}{\mathbb Z}     

\renewcommand{\a}{\alpha}
\renewcommand{\b}{\beta}

\renewcommand{\d}{\delta}

\newcommand{\e}{\varepsilon}

\newcommand{\w}{\omega}

\renewcommand{\l}{\lambda}
\renewcommand{\L}{\Lambda}

\newcommand{\s}{\sigma}

\renewcommand{\k}{\kappa}

\renewcommand{\P}{\mathbb{P}_{\eta}}   

\newcommand{\E}{\mathbb{E}_{\eta}}   

\renewcommand{\th}{\theta}

\newcommand{\fl}[1]{ \left\lfloor #1 \right\rfloor }
\renewcommand{\cl}[1]{ \left\lceil #1 \right\rceil }
\DeclareMathOperator{\supp}{supp}

\newcommand{\ind}[1]{ \mathbf{1}_{ \{ #1 \} } } 
\newcommand{\one}{\mathbf{1}}

\newcommand{\lcrit}{\l_{\text{crit}}}
\newcommand{\tcrit}{t^*}
\newcommand{\ts}{t_*}
\newcommand{\vp}{\mathrm{v}_0}

\title[Large deviations for RWRE on a strip]{Large deviations for random walks in a random environment on a strip}
\date{\today}
\author{Jonathon Peterson}
\address{Jonathon Peterson \\  Purdue University \\ Department of Mathematics \\ 150 N University Street \\ West Lafayette, IN  47907 \\ USA}
\email{peterson@math.purdue.edu}
\urladdr{http://www.math.purdue.edu/~peterson}

\begin{document}

\begin{abstract}
 We consider a random walk in a random environment (RWRE) on the strip of finite width $\Z \times \{1,2,\ldots,d\}$. 
We prove both quenched and averaged large deviation principles for the position and the hitting times of the RWRE. 
Moreover, we prove a variational formula that relates the quenched and averaged rate functions, thus extending a result of Comets, Gantert, and Zeitouni for nearest-neighbor RWRE on $\Z$ 
\end{abstract}

\maketitle

\section{Introduction}

In this paper we will study the large deviations of a random walk in a random environment (RWRE) on the strip $\Z \times [d]$ where we use the notation $[d]$ to denote the finite set $\{1,2,\ldots,d\}$.
A point $(k,i) \in \Z \times [d]$ will be said to be at \emph{height} $i$ of \emph{level} $k$ in the strip. We will be interested in RWRE on the strip that can move at most one level to the left or right. In this case, the model of RWRE on the strip can be described as follows.
An \emph{environment} $\w$ is given by three sequences of $d\times d$ matrices. That is $\w = \{ \w_n \}_{n \in \Z} = \{(q_n,r_n,p_n)\}_{n \in \Z}$, where $q_n,r_n$ and $p_n$ are non-negative $d\times d$ matrices for each $n$ such that $q_n+r_n+p_n$ is a stochastic matrix for every $n \in \Z$. That is,
\[
 \sum_{j\in[d]} \left( q_n(i,j) + r_n(i,j) + p_n(i,j) \right) = 1, \quad \forall i \in [d], \, n \in \Z.
\]
For a fixed environment $\w$, we can define the RWRE starting at $(x,i) \in \Z \times [d]$ to be the Markov chain $\xi_n$ with distribution $P_\w^{(x,i)}$ defined by
$P_\w^{(x,i)}( \xi_0 = (x,i) ) = 1$ and
\begin{equation}\label{eq:transprob}
 P_\w^{(x,i)}\left( \xi_{n+1} = (m,j) \, | \, \xi_n = (k,i) \right)
=
\begin{cases}
 q_k(i,j) & \text{if } m=k-1 \\
 r_k(i,j) & \text{if } m=k \\
 p_k(i,j) & \text{if } m=k+1 \\
 0 & \text{otherwise}.
\end{cases}
\end{equation}
That is, the matrices $q_k,r_k$ and $p_k$ give the one-step transition probabilities for jumping to the level to the left of level $k$, within level $k$, and to the right of level $k$, respectively.
$P_\w^{(x,i)}$ is called the quenched law of the RWRE, and expectations with respect to this law are denoted $E_\w^{(x,i)}$.

We can also define the averaged law of the RWRE by first choosing the environment randomly.
To make this precise, let
\[
 \Sigma = \{(q,r,p) \in \R_+^{d\times d} \times \R_+^{d\times d} \times \R_+^{d\times d} \, : \, (q+r+p)\one = \one \}
\]
denote the set of all transition probabilities for a fixed level of the strip so that $\Omega = \Sigma^\Z$ is the set of all possible environments $\w$ on the strip.
Let $\mathcal{F}$ be the Borel $\s$-algebra on $\Omega$, and let $\eta$ be a probability measure on $(\Omega,\mathcal{F})$.
Then the averaged law of the RWRE is defined by
\[
 \P^{(x,i)}(\cdot) = E_\eta\left[ P_\w^{(x,i)}(\cdot) \right]
\]
where $E_\eta$ denotes expectation with respect to the distribution $\eta$ on the random environment $\w$.
Expectations with respect to the averaged law $\P^{(x,i)}$ of the RWRE will be denoted $\E^{(x,i)}$.

Often times we will start the RWRE at a location $(0,i)$ in level 0 of the strip. However, we may also chose to start the RWRE at a random height $i \in [d]$ instead. To this end, for a fixed environment $\w$ and a probability distribution $\pi$ on $[d]$ we will define $P_\w^\pi(\cdot) = \sum_{i} \pi(i) P_\w^{(0,i)}(\cdot)$. That is, the RWRE starts at $(0,i)$ with probability $\pi(i)$. The corresponding averaged law will be denoted $\P^\pi$. At times we will even allow $\pi = \pi(\w)$ to depend on the environment $\w$ in a measurable way so that $\P^\pi(\cdot) = \int_\Omega P_\w^\pi(\cdot) \, \eta(d\w)$ is still well defined. Naturally, $E_\w^\pi$ and $\E^\pi$ will denote the corresponding quenched and averaged expectations.

The first results for RWRE on a strip were by Bolthausen and Goldsheid \cite{bgSRT} who gave a criterion for the RWRE to be recurrent or transient to the left/right. 
Subsequently, Goldsheid proved a law of large numbers and a quenched central limit theorem \cite{gSCLT} and independently Roiterstein also proved a law of large numbers and an averaged central limit theorem for the RWRE \cite{rSCLT}. 
We note that since the strip is bounded in the second coordinate, the law of large numbers and the central limit theorems are proved for the first coordinate of the RWRE. That is, if we write $\xi_n = (X_n,Y_n)$ then the law of large numbers proved in \cite{gSCLT,rSCLT} states that there exists a $\vp \in [-1,1]$ such that
\begin{equation}\label{llnXn}
 \lim_{n\ra\infty} \frac{X_n}{n} = \vp, \quad \P^\pi\text{ - a.s.}
\end{equation}
In both \cite{gSCLT} and \cite{rSCLT}, this law of large numbers for $X_n/n$ was deduced from a law of large numbers for hitting times. 
For any $x \in \Z$ let $T_x = \inf \{ n\geq 0: X_n = x \}$ be the hitting time of the level $\{x \} \times [d]$ for the random walk. It was shown in \cite{gSCLT,rSCLT} for RWRE on the strip that are recurrent or transient to the right that 
\begin{equation}\label{llnTn}
 \lim_{n\ra\infty} \frac{T_n}{n} = 1/\vp, \quad \P^\pi\text{ - a.s.},
\end{equation}
where we interpret $1/\vp = \infty$ if $\vp=0$. 
The main goal of this paper is to study asymptotics of probabilities of large deviations away from the laws of large numbers in \eqref{llnXn} and \eqref{llnTn}. That is, we will prove large deviation principles for both $T_n/n$ and $X_n/n$ under both the quenched and averaged laws. 
Moreover, we will give a variational formula relating the respective quenched and averaged large deviation rate functions.

Large deviations of RWRE on $\Z$ and $\Z^d$ have been studied previously using a number of different approaches \cite{gdhLDP,cgzLDP,vLDP,zLEQLDP,rLDPmixing,yQLDP,yALDP,pzALDRF}.
Some of these results on $\Z^d$ allow for bounded step sizes \cite{vLDP,yQLDP,yALDP} and thus would apply to certain RWRE on the strip 
(see Appendix \ref{sec:boundedjumps} below). 
However, in these papers the quenched and averaged large deviation principles are proved using different approaches and thus it is very difficult to compare the quenched and averaged rate functions.
In contrast, \cite{cgzLDP} develops a unified approach for studying both quenched and averaged large deviations of nearest-neighbor RWRE on $\Z$ that not only proves quenched and averaged large deviation principles for $X_n/n$ and $T_n/n$ but also gives a variational expression relating the quenched and averaged rate functions.

In this paper we will adapt the approach of \cite{cgzLDP} to the case of RWRE on the strip. 
We note that the results in \cite{cgzLDP} were later generalized in \cite{dgzLDPH} to the model of nearest-neighbor RWRE on $\Z$ with holding times, and at times we will borrow from ideas in \cite{dgzLDPH} as well. 
The main difference from the one-dimensional case is that the random walk can enter a new level at any height and thus the differences of hitting times $T_k - T_{k-1}$ are no longer independent under the quenched measure which makes it difficult to study the asymptotics of the quenched moment generating functions $E_\w^\pi[ e^{\l T_n} \ind{T_n < \infty} ]$. This is overcome by keeping track of both the hitting times and the heights at which the random walk enters a level so that $E_\w^\pi[ e^{\l T_n} \ind{T_n < \infty} ]$ can be represented using products of random ($\w$-dependent) matrices.
We then use some ideas from \cite{fkPRM} to obtain formulas for the asymptotics of these products.
As in \cite{cgzLDP}, understanding the asymptotics of the moment generating functions of the hitting times is the key to deriving large deviation principles for both the hitting times and the speed of the random walk.

\subsection{Main Results}

Before stating our main results, we need to first introduce some assumptions on the environment.
Our first assumption is that the distribution on environments is spatially ergodic with respects to shifts of the $\Z$-coordinate.
Recalling that $\w_n = (q_n,r_n,p_n)$, 
let $\theta: \Omega \ra \Omega$ be the natural left shift operator defined by $(\theta \w)_k = \w_{k+1}$.
\begin{asm}\label{asmerg}
 The sequence $\{ \w_n \}_{n\in\Z}$ is stationary and ergodic under the measure $\eta$ on environments. That is, the dynamical system $(\Omega,\mathcal{F},\eta,\theta)$ is stationary and ergodic.
\end{asm}

For technical reasons we will also need some strong ellipticity assumptions on the environments. 
\begin{defn}
For any $\k>0$, let $\Sigma_\k \subset \Sigma$ be the set of transition probabilities $(q,r,p)$ from a given level that satisfy
\begin{equation}\label{onestepellipticity}
 \sum_j q(i,j) \geq \k, \quad \text{and} \quad \sum_j p(i,j) \geq \k, \quad \forall i \in [d]
\end{equation}
and
\begin{equation}\label{entryheightellipticity}
 ((I - r)^{-1}q)(i,j) \geq \k , \quad \text{and} \quad ((I - r)^{-1}p)(i,j) \geq \k, \quad \forall i,j \in [d]. 
\end{equation}
Moreover, define $\Omega_\k = \Sigma_\k^\Z$ so that environments $\w \in \Omega_\k$ satisfy the uniform ellipticity assumptions \eqref{onestepellipticity} and \eqref{entryheightellipticity} at every level in the strip. 
\end{defn}

\begin{asm}\label{easm}
 There exists a $\k > 0$ such that $\eta( \w \in \Omega_\k ) = 1$. 
\end{asm}
\begin{rem}
 The uniform ellipticity assumptions \eqref{onestepellipticity} and \eqref{entryheightellipticity} have simple probabilistic interpretations. 
For any environment $\w \in \Omega_\k$,
\[
 P_\w^{(k,i)}( X_1 = k-1 ) \geq \k, \quad  P_\w^{(k,i)}( X_1 = k+1 ) \geq \k, \quad \forall k \in \Z, \, i \in [d],
\]
and 
\[
 P_\w^{(k,i)}( T_{k-1} < T_{k+1}, \, Y_{T_{k-1}} = j ) \geq \k ,  \quad P_\w^{(k,i)}( T_{k+1} < T_{k-1}, \, Y_{T_{k+1}} = j ) \geq \k , \quad \forall k \in \Z, \, i,j \in [d].
\]
Also, note that for any environment $\w \in \Omega_\k$ the state space $\Z \times [d]$ for the random walk is irreducible. 
\end{rem}
\begin{rem}
 Assumption \ref{easm} is slightly stronger than the uniform ellipticity assumptions made by Goldsheid in \cite{gSCLT} where the condition \eqref{onestepellipticity} is replaced by the assumption that $\sum_j (p(i,j) + q(i,j)) \geq \k$ for every $i\in[d]$. 
\end{rem}

As in \cite{cgzLDP}, the key to proving the quenched large deviation principle will be the derivation of a formula for the asymptotic quenched logarithic moment generating function of $T_n$. In particular, we will show that there is a deterministic function $\L_\eta(\l)$ such that
\begin{equation}\label{eq:qlmgflim}
 \lim_{n\ra\infty} \frac{1}{n} \log E_\w^\pi\left[ e^{\l T_n} \ind{T_n<\infty} \right] = \L_\eta(\l), \quad \eta\text{ - a.s.},
\end{equation}
where the limit does not depend on the initial distribution $\pi$ for the starting height. In Section \ref{sec:qlmgf} we will prove the existence of the above limit, give a formula for $\L_\eta(\l)$, and show that $\L_\eta(\l)$ is differentiable. From this, the following quenched large deviation principle follows in the standard way.
\begin{thm}\label{QLDPTn}
 For a distribution $\eta$ on environments satisfying Assumptions \ref{asmerg} and \ref{easm}, define
\begin{equation}\label{Jetadef}
 J_\eta(t) = \sup_\l \{ \l t - \L_\eta(\l) \}.
\end{equation}
Then, for $\eta$-a.e.\ environment $\w$ and any initial distribution $\pi$ for the starting height at level 0 (even depending on the environment), $T_n/n$ satisfies a weak large deviation principle under the measure $P_\w^\pi$ with rate function $J_\eta$. That is, for any open $G $
\[
 \liminf_{n\ra\infty} \frac{1}{n} \log P_\w^\pi( T_n/n \in G ) \geq - \inf_{t \in G} J_\eta(t),
\]
and for any compact $F$
\begin{equation}\label{eq:qldpub}
 \limsup_{n\ra\infty} \frac{1}{n} \log P_\w^\pi( T_n/n \in F ) \leq - \inf_{t \in F} J_\eta(t).
\end{equation}
\end{thm}
\begin{rem}
Recall that a (strong) large deviation principle means that the large deviation upper bound holds for all closed $F$ as well.
 We can only claim a \emph{weak} large deviation principle since it may be the case that $\lim_{t\ra\infty} J_\eta(t) = \inf_t J_\eta(t) > 0$. However, we will show in Lemma \ref{qrfpropTn} below that if the random walk is recurrent or transient to the right then $\inf_t J_\eta(t)=0$, and thus in these cases Theorem \ref{QLDPTn} can easily be strengthened to a full large deviation principle under $P_\w^\pi$.
\end{rem}

To prove the averaged large deviation principle for the hitting times we will need a more restrictive assumption on the distribution $\eta$ on environments.
Let $M_1(\Omega_\k)$ be the set of probability distributions on the set of environments $\Omega_\k$, and let $M_1^s(\Omega_\k)$ (and $M_1^e(\Omega_\k)$) denote the set of stationary (ergodic) distributions $\eta$ on environments; that is, $\{\w_n\}$ is a stationary (ergodic) under $\eta$.

\begin{asm}\label{asmplldp}
 The distribution $\eta \in M_1^e(\Omega_\k)$ satisfies a process level large deviation principle on the space $M_1(\Omega_\k)$ equipped with the topology of weak convergence of probability measures. That is,
\begin{equation}\label{eq:Lndef}
 L_n = \frac{1}{n} \sum_{k=0}^{n-1} \d_{\th^k \w}
\end{equation}
satisfies a large deviation principle with rate function $h(\cdot|\eta)$, the specific relative entropy with respect to $\eta$.
\end{asm}

We will say that a distribution $\eta \in M_1^s(\Omega)$ on environments is locally equivalent to the product of its marginals if for all $n\geq 1$ the joint distribution of $(\w_1,\w_2,\ldots,\w_n)$ is absolutely continuous with respect to the product measure $\eta_0^n$, where $\eta_0$ is the marginal of $\w_0$ under the measure $\eta$.

\begin{asm}\label{asmlocal}
 The distribution $\eta$ on environments is locally equivalent to the product of its marginals, and for every stationary measure $\a \in M_1^s(\Omega_\k)$ there exists a sequence $\a_n \in M_1^e(\Omega_\k)$ of ergodic measures such that $\a_n \ra \a$ (in the topology of weak convergence of probability measures) and $h(\a_n|\eta) \ra h(\a|\eta)$.
\end{asm}

\begin{rem}
 Assumptions \ref{asmplldp} and \ref{asmlocal} were also made for the averaged large deviation principles in in \cite{cgzLDP} and \cite{dgzLDPH}. 
Also, it is known that Assumptions \ref{asmplldp} and \ref{asmlocal} are satisfied if $\eta$ is a Gibbs measure on $\Omega_\k$ with summable, translation invariant interaction potential (see Theorem 4.1 and Lemma 4.8 in \cite{fRFDP}). In particular, Assumptions \ref{asmplldp} and \ref{asmlocal} hold if $\eta$ is an i.i.d.\ measure on environments. 
\end{rem}

Our next main result is a large deviation principle for the hitting times under the averaged measure. 
It is intuitively clear that the averaged rate function should be less than or equal to the corresponding quenched rate function. This is because under the quenched measure large deviations occur due to atypical behavior from the random walk, but under the averaged measure large deviations can occur due to some combination of the choice of an atypical environment and the walk doing some atypical behavior (thus making the averaged large deviation probabilities decay more slowly). The following Theorem extends the variational formula in \cite{cgzLDP} relating the quenched and averaged rate functions that makes the above intuition precise. 

\begin{thm}\label{th:aldpTn}
 Let the distribution on environments $\eta$ satisfy Assumptions \ref{easm}, \ref{asmplldp} and \ref{asmlocal}. Then, for any initial distribution $\pi$ the hitting times $T_n/n$ satisfy a weak large deviation principle under the measure $\P^\pi$ with 
convex, lower semicontinuous 
rate function
\begin{equation}\label{eq:arfTn}
 \mathbb{J}_\eta(t) := 
\inf_{\a \in M_1^e(\Omega_\k)} ( J_\a(t) + h(\a|\eta) ).
\end{equation}
\end{thm}

Continuing to follow the approach of \cite{cgzLDP}, we will use Theorems \ref{QLDPTn} and \ref{th:aldpTn} to deduce quenched and averaged large deviation principles for $X_n/n$. In order to do this we will need not only a large deviation principle for $T_n/n$ but also for $T_{-n}/n$. However, since we have made no assumptions about recurrence or transience in the statements of Theorems \ref{QLDPTn} and \ref{th:aldpTn}, an obvious symmetry argument implies large deviation principles for $T_{-n}/n$ as well. 
To make this precise we introduce the following notation. 
\begin{defn}
For any environment $\w \in \Omega$ let $\w^{\text{Inv}}$ denote the environment induced by reflecting the strip $\Z \times [d]$ in the first coordinate. That is,
\[
 ( q_n(\w^{\text{Inv}}), r_n(\w^{\text{Inv}}),p_n(\w^{\text{Inv}})) = ( p_{-n}(\w), r_{-n}(\w), q_{-n}(\w) ), \quad \forall n \in \Z.
\]
Moreover, for any distribution $\eta \in M_1(\Omega)$, let $\eta^{\text{Inv}}$ be the induced distribution on $\w^{\text{Inv}}$.
\end{defn}
With this notation it is clear that Theorems \ref{QLDPTn} and \ref{th:aldpTn} imply quenched and averaged large deviation principles for $T_{-n}/n$ with rate functions $J_{\eta^{\text{Inv}}}$ and $\mathbb{J}_{\eta^{\text{Inv}}}$, respectively. 
We are now ready to state the quenched and averaged large deviation principles for the speed $X_n/n$ of the RWRE. 

\begin{thm}\label{th:qldpXn}
 Let the distribution on environments $\eta$ satisfy Assumptions \ref{asmerg} and \ref{easm}.
Then, for any initial distribution $\pi$ (even depending on the environment) $X_n/n$ satisfies a large deviation principle under the measure $P_\w^\pi$ with deterministic, lower semicontinuous, convex rate function
\[
 I_\eta(x)
=
\begin{cases}
 x J_\eta(1/x) & x>0 \\
 \lim_{t \ra\infty} J_\eta(t)/t & x = 0 \\
 |x| J_{\eta^{\text{Inv}}}(1/|x|) & x<0.
\end{cases}
\]
\end{thm}
\begin{rem}
It is not clear from the formula above that $I_\eta(x)$ is continuous at $x=0$. However, part of the proof of Theorem \ref{th:qldpXn} will be to show that $\lim_{t \ra\infty} J_\eta(t)/t = \lim_{t \ra\infty} J_{\eta^{\text{Inv}}}(t)/t$. 
\end{rem}

Our final main result is the corresponding averaged large deviation principle for $X_n/n$.

\begin{thm}\label{th:aldpXn}
 Let the distribution on environments $\eta$ satisfy Assumptions \ref{easm} - \ref{asmlocal}.
Then, for any initial distribution $\pi$, $X_n/n$ satisfies a large deviation principle under the measure $\P^\pi$ with lower semicontinuous rate function
\[
 \mathbb{I}_\eta(x)
=
\begin{cases}
 x \mathbb{J}_\eta(1/x) & x>0 \\
 I_\eta(0) & x = 0 \\
 |x| \mathbb{J}_{\eta^{\text{Inv}}}(1/|x|) & x<0.
\end{cases}
\]
\end{thm}
\begin{rem}
 Again, part of the proof of Theorem \ref{th:qldpXn} will be showing that $\mathbb{I}_\eta$ as defined above is continuous at $x=0$.
Naturally, the variational formula \eqref{eq:arfTn} implies a corresponding variational formula for $\mathbb{I}_\eta(x)$. The only difficulty is proving the variational formula at $x=0$; see \eqref{eq:varformI} below. 
\end{rem}

The structure of the paper is as follows. In Section \ref{sec:qmgf} we introduce the matrices $\Phi_k(\l)$ which are quenched moment generating functions for hitting times that also take into account the height at which the random walk enters a level. 
Then, in Section \ref{sec:qlmgf} we use these matrices to compute the asymptotic log moment generating function $\L_\eta(\l)$ for hitting times as well as a formula for $\L_\eta'(\l)$. 
In Sections \ref{sec:qldpTn} and \ref{sec:aldpTn} we  prove the quenched and averaged large deviation principles for hitting times, and then in Section \ref{sec:ldpXn} we show how to transfer these to obtain quenched and averaged large deviation principles for $X_n/n$. 
We conclude the paper with a short appendix on RWRE on $\Z$ with bounded jumps. It is well known that such RWRE with bounded jumps can be interpreted as a special case of RWRE on a strip. In Appendix \ref{sec:boundedjumps} we examine a natural class of RWRE with bounded jumps to which the results in this paper apply, and we show how to modify the proofs in this paper to RWRE with bounded jumps that do not satisfy the strong uniform ellipticity assumptions in Assumption \ref{easm}.

\subsection{Notation}
Before beginning with the proofs of the main results, we will introduce some notation that will be used throughout the paper.
For vectors $\textbf{x} = (x_1,\ldots, x_d) \in \R^d$, let $\|x\|_\infty = \max_i |x_i|$  and $\|x\|_1 = \sum_i |x_i|$ be the standard $\ell^\infty$ and $\ell^1$ norms, respectively. Also, let a norm on $d\times d$ matrices be given by
\[
 \|A\| = \max_i \sum_j |A_{i,j}|.
\]
If the entries of $A$ are non-negative then $\|A\| = \max_i e_i A \mathbf{1}$, where the vectors $\{e_i\}_{i \in [d]}$ are the standard basis vectors for $\R^d$. We will use this fact without mention at several points throught the paper.
Note that matrix norm defined in this way is the $\ell^1$ operator norm acting on row vectors to the left and the $\ell^\infty$ operator norm acting on column vectors to the right. In particular, for any row vector $\mathbf{x}$ and column vector $\mathbf{y}$ we have $| \mathbf{x} A \mathbf{y} | \leq \|\mathbf{x} \|_1 \|A\| \| \mathbf{y} \|_\infty$.

\noindent\textbf{Acknowledgement} 
We would like to thank to Alex Roiterstein for originally suggesting this problem and for many stimulating discussions on RWRE on the strip. Also, many thanks to 
Ofer Zeitouni for several discussions on large deviation generalities that were helpful in the course of preparing this paper.

\section{Quenched moment generating functions}\label{sec:qmgf}

As mentioned in the introduction, they key point in proving the quenched large deviation principle for the hitting times is to prove the limit \eqref{eq:qlmgflim} which gives the exponential asymptotics of the quenched moment generating functions of the hitting times. In this section, we will prepare the foundation for proving the limit in \eqref{eq:qlmgflim} by introducing some useful notation and deriving some uniform upper and lower bounds on the quenched moment generating functions.

For any environment $\w$, $\l \in \R$ and $k\in\Z$ let $\Phi_k(\l)$ be the $d\times d$ matrix with entries given by
\[
 \Phi_{k}(\l)(i,j) := E_\w^{(k,i)} \left[ e^{\l T_{k+1} } \ind{ T_{k+1}< \infty, \: Y_{T_{k+1}}=j } \right], \quad i,j \in [d].
\]

\begin{lem}\label{lcritlem}
There exists a constant $\lcrit=\lcrit(\eta) \geq 0$ such that if $\l< \lcrit$ then $\Phi_k(i,j)(\l) < \infty$ for all $i,j \in[d]$ and $k \in \Z$, $\eta$-a.s., and if $\l > \lcrit$ then $\Phi_k(i,j)(\l) = \infty$ for all $i,j \in [d]$, and $k\in \Z$, $\eta$-a.s.
\end{lem}

\begin{proof}
Obviously, $\Phi_0(\l)(i,j) < \infty$ for all $i,j \in [d]$ if $\l\leq 0$. Thus, we only need to consider $\l > 0$, in which case Assumptions \ref{easm} implies that 
\begin{align}
 \Phi_0(\l)(j,k) 
&\geq E_\w^{(0,j)}\left[ e^{\l T_1} \ind{ T_{-1} < T_1 < \infty, \, Y_{T_{-1}} = i, \, Y_{T_1} = k } \right]  \nonumber \\
&\geq \k e^\l \sum_{l \in [d]} \Phi_{-1}(\l)(i,l) \Phi_0(\l)(l,k) \label{Nstep1} \\
&\geq \k^2 e^{2\l} \sum_{l \in [d]} \Phi_{-1}(\l)(i,l), \quad \forall i,j,k \in[d]. \label{Nstep}
\end{align}
It follows that $\| \Phi_0(\l) \| < \infty$ implies that $\| \Phi_{-1}(\l) \| < \infty$ and thus 
\[
 \{ \| \Phi_0(\l) \| < \infty \} = \bigcap_{n \leq 0} \{ \| \Phi_n(\l) \| < \infty \}, \quad \eta\text{ - a.s.}
\]
Since $\bigcap_{n \leq 0} \{ \| \Phi_n(\l) \| < \infty \}$ is invariant under (right) shifts of the environment and the distribution $\eta$ on environments is ergodic, we can conclude that $\eta( \|\Phi_0(\l) \| < \infty) \in \{0,1\}$ for any $\l \in \R$. 
Define $\lcrit = \lcrit(\eta) := \sup \{ \l : \eta( \| \Phi_0(\l) \| < \infty ) = 1 \} $. By the monotonicity of $\|\Phi_0(\l)\|$ in $\l$ we have that $\|\Phi_0(\l)\| < \infty$ for all $\l < \lcrit$ and $\eta$-a.e.\ environment $\w$. Since $\eta$ is shift invariant we also have $\|\Phi_n(\l)\| < \infty$ for all $n \in \Z$, $\l < \lcrit$, $\eta$-a.s.
On the other hand, if $\l > \lcrit$, then $\| \Phi_{-1}(\l) \| = \infty$, $\eta$-a.s. 
However, maximizing \eqref{Nstep} over $i$ we obtain that $\Phi_0(\l)(j,k) \geq \k^2 e^{2\l} \|\Phi_{-1}(\l)\| = \infty$ for any $j,k \in [d]$. 
Again, since $\eta$ is shift invariant, this implies that $\Phi_n(j,k)(\l) = \infty$ for all $n \in \Z$, $j,k \in[d]$, $\eta$-a.s.
\end{proof}
\begin{rem}\label{lcritrem}
 Note that the above lemma does not say whether or not $\|\Phi_n(\l)\|$ is finite when $\l = \lcrit$. However, it will follow from the proof of Lemma \ref{PhiUboundlem} below that $\|\Phi_n(\lcrit)\|<\infty$ for all $n$. In fact, Lemma \ref{PhiUboundlem} will even give a uniform upper bound on the entries of $\Phi_n(\lcrit)$. 
\end{rem}

Next, we would like to prove upper and lower bounds on the entries of $\Phi_k(\l)$ when $\l \leq \lcrit$. To this end, we first need the following Lemma which follows easily from the uniform ellipticity assumptions on the environment. 

\begin{lem}\label{lem:fse}
For any $\k \in (0,1/2)$ there exists an integer $N_\k<\infty$ such that for all  $\w \in \Omega_\k$
\[
 P_\w^{(k,i)}( T_{k+1} \leq N_\k, \, Y_{T_{k+1}} = j ) \geq \k/2, 
\quad
P_\w^{(k,i)}( T_{k-1} \leq N_\k, \, Y_{T_{k-1}} = j ) \geq \k/2, 
\quad \forall k \in \Z, \, i,j \in [d]. 
\]
\end{lem}
\begin{proof}
Obviously it is enough to prove the lower bounds when starting in level $k=0$. 
For a random walk started at a point $(0,i)$ in level 0 of the strip, let $\tau = T_1 \wedge T_{-1}$ be the exit time of level 0. Conditions \eqref{onestepellipticity} and \eqref{entryheightellipticity} imply that
\begin{equation}\label{tauasm1}
 P_\w^{(0,i)}( \tau > 1 ) \leq 1 - 2\k, \quad \forall i \in [d],
\end{equation}
and
\begin{equation}\label{tauasm2}
 P_\w^{(0,i)}( \xi_\tau = (1,j) ) \geq \k,  \quad P_\w^{(0,i)}( \xi_\tau = (-1,j) ) \geq \k, \quad \forall i,j \in [d].
\end{equation}
Note that iterating the lower bound \eqref{tauasm1} implies that $P_\w^{(0,i)}( \tau > N ) < (1-2\k)^N$ for any non-negative integer $N$. 
Therefore, 
\begin{align*}
 P_\w^{(0,i)}( T_1 \leq N, \, Y_{T_1} = j ) 
&\geq P_\w^{(0,i)}( \xi_\tau = (1,j) ) - P_\w^{(0,i)}( \tau > N ) \\
&\geq \k - (1-2\k)^N. 
\end{align*}
Similarly, $P_\w^{(0,i)}(  T_{-1} \leq N, \, Y_{T_{-1}} = j ) \geq \k - (1-2\k)^N$, and letting $N_\k = \lceil \log(\k/2)/\log(1-2\k) \rceil$ completes the proof of the lemma. 
\end{proof}

The following lemma which gives uniform upper and lower bounds on the entries of $\Phi_k(\l)$ will be crucial throughout the remainder of the paper. 
\begin{lem}\label{PhiUboundlem}
 If $\eta$ satisfies Assumptions \ref{asmerg} and \ref{easm},
then for any $\l \leq \lcrit$ there exists a constant $c_\l \in (0,1]$ (depending only on $\l$ and the choice of $\k$ in Assumption \ref{easm}) such that
\begin{equation}\label{PhiUbound}
 \eta\left( c_\l \leq \Phi_0(\l)(i,j) \leq \frac{1}{c_\l}, \, \forall i,j \in [d] \right) = 1.
\end{equation}
If in addition $\eta$ is locally equivalent to the product of its marginals and we denote by $\Sigma_\eta$ the support of $\w_0 = (q_0,r_0,p_0)$ under $\eta$, then it follows that 
\begin{equation}\label{eq:PhiUboundU}
 c_\l \leq \Phi_0(\l,\w)(i,j) \leq \frac{1}{c_\l}, \quad \forall \w \in \Sigma_\eta^\Z, \,   i,j \in [d], \, \l \leq \lcrit(\eta). 
\end{equation}
\end{lem}
\begin{proof}
For the remainder of the proof, fix a $\k>0$ that satisfies Assumption \ref{easm} (i.e., $\eta(\w \in \Omega_\k) = 1$). 
First we prove the almost sure upper and lower bounds on $\Phi_0(\l)(i,j)$ in \eqref{PhiUbound}. 
For the lower bound, since $e^{\l T_1} \geq (e^{\l N_\k} \wedge 1)$ on the event $\{T_1 \leq N_\k\}$ we have that
\begin{equation}\label{eq:Ulb}
\begin{split}
 \Phi_0(\l)(i,j) &\geq E_\w^{(0,i)}\left[ e^{\l T_1} \ind{ T_1 \leq N_\k, \, Y_{T_1} = j } \right] \\
&\geq \left( e^{\l N_\k} \wedge 1 \right) P_\w^{(0,i)}( T_1 \leq N_\k, \, Y_{T_1} = j ) \geq \left( e^{\l N_k} \wedge 1 \right) \k/2,
\end{split}
\end{equation}
where the last inequality follows from Lemma \ref{lem:fse}.
For an upper bound, first note that $\Phi_0(\l)(i,j) \leq 1$ if $\l \leq 0$. Thus, we only need to prove a uniform upper bound when $\l \in (0,\lcrit]$.
To this end, note that \eqref{Nstep1} implies that 
\[
 \Phi_0(\l)(j,k) \geq \k e^{\l} \Phi_{-1}(\l)(i,j) \Phi_0(\l)(j,k), \quad \forall i,j,k \in [d]. 
\]
If $\l < \lcrit$ then $\Phi_0(\l)(j,k) \in (0,\infty)$ and we may cancel these terms from both sides of the above inequality to obtain that $\Phi_{-1}(\l)(i,j) \leq (1/\k) e^{-\l} \leq 1/\k$.
Since the law $\eta$ on environments is shift invariant, the same uniform upper bound holds for $\Phi_0(\l)(i,j)$ with probability one. Finally, since $\Phi_0(\l)(i,j) \leq (1/\k)$ for all $\l < \lcrit$, the monotone convergence theorem implies that $\Phi_0(\lcrit)(i,j) \leq (1/\k)$, $\eta$-a.s. This completes the proof of \eqref{PhiUbound} with 
$c_\l = \frac{\k}{2} (e^{\l N_{\k}} \wedge 1 )$. 

Moving now to the proof of \eqref{eq:PhiUboundU}, note that the argument above giving the lower bound on the entries of $\Phi_0(\l)$ only depends on the fact that $\w \in \Omega_\k$, and since obviously $\Sigma_\eta^\Z \subset \Omega_\k$ the lower bound in \eqref{eq:PhiUboundU} also holds. 
To prove the upper bound in \eqref{eq:PhiUboundU} we first introduce some notation. For any $M<\infty$ let $\Phi_{k,M}(\l)$ be the ``truncated'' quenched moment generating functions defined by 
\begin{equation}\label{PhikMdef}
 \Phi_{k,M}(\l)(i,j) = E_\w^{(k,i)}\left[ e^{\l T_{k+1}} \ind{T_{k+1} \leq M, \, Y_{T_{k+1}} = j } \right].
\end{equation}
It is easy to see that $\Phi_{k,M}(\l)$ depends on the environment $\w$ through $\w_{(-M,0]}$ where we use the notation $\w_{(k,\ell]} := (\w_{k+1},\w_{k+2},\ldots,\w_\ell) \in \Sigma^{\ell-k}$ for any environment $\w$ and any $k<\ell$. 
Moreover, the function $\w \mapsto \Phi_{0,M}(\l,\w)(i,j)$ from $\Omega_\k \ra \R$ is continuous for any fixed $M<\infty$ and $i,j \in[d]$ since the quenched expectation in \eqref{PhikMdef} can be expressed as a sum over finitely many paths. 
Thus, for any fixed $M<\infty$, $i,j \in [d]$ and $\l \leq \lcrit(\eta)$ the set 
\[
 G_{M,i,j,\l} = \{ \w \in \Omega_\k: \, \Phi_{0,M}(\l,\w)(i,j) > 1/c_\l \}
\]
is an open subset of $\Sigma^\Z$. 

Fix $\l \leq \lcrit(\eta)$, and assume for contradiction that there 
exists $\hat\w \in \Sigma_\eta^\Z$ and $i,j \in [d]$ with $\Phi_0(\l,\hat\w)(i,j) > 1/c_\l$. 
Then by monotone convergence, $\hat\w \in G_{M,i,j,\l}$ for all $M$ large enough. 
Since the open set $G_{M,i,j,\l}$ intersects $\Sigma_\eta^\Z$ and since $G_{M,i,j,\l}$ is $\s(\w_{(-M,0]})$-measurable, it follows that $(\eta|_0)^M( G_{M,i,j,\l} ) > 0$. However, since $\eta$ is locally equivalent to the product of its marginals this implies that $\eta(G_{M,i,j,\l}) = \eta|_{(-M,0]}(G_{M,i,j,\l}) > 0$ as well. Thus, with $\eta$-positive probability $1/c_\l < \Phi_{0,M}(\l)(i,j) \leq \Phi_0(\l)(i,j)$. Since this contradicts \eqref{PhiUbound}, this proves that the upper bound \eqref{eq:PhiUboundU} does indeed hold. 
\end{proof}

\begin{rem}
 Note that the above proof shows that $\k/2 \leq \Phi_0(\l)(i,j) \leq (1/\k)e^{-\l}$ for all $\l \in [0,\lcrit]$. 
A priori there is nothing preventing $\lcrit$ from being infinite. However, since $\k/2 \leq (1/\k)e^{-\l}$ for $\l \in (0,\lcrit]$ we can conclude that $\lcrit \leq -\log(\k^2/2)$. 
\end{rem}

\begin{rem}\label{rem:PhikMUbounds}
 It will be important below to note that (for $M$ large enough) we can give uniform upper and lower bounds on the entries of the truncated moment generating functions $\Phi_{k,M}(\l)$ as well. It is obvious from the definitions that $\Phi_{k,M}(\l)(i,j) \leq \Phi_k(\l)(i,j)$ so that the same uniform upper bound holds for any $M<\infty$. Moreover, the argument in \eqref{eq:Ulb} giving the uniform lower bound on the entries of $\Phi_k(\l)$ gives the same lower bound on the entries of $\Phi_{k,M}(\l)$ if $M \geq N_\k$. That is, 
\begin{equation}\label{eq:PhikMUbounds}
 c_\l \leq \Phi_{0,M}(\l)(i,j) \leq \frac{1}{c_\l}, \quad \forall \w \in \Sigma_\eta^\Z, \, i,j \in [d], \, \l \leq \lcrit(\eta), \text{ and } M \geq N_\k. 
\end{equation}
\end{rem}

\section{The quenched logarithmic moment generating function for hitting times}\label{sec:qlmgf}

In this section we will prove that the limit $\lim_{n\ra\infty} \frac{1}{n} \log E_\w^\pi[ e^{\l T_n} \ind{T_n < \infty} ]$ exists almost surely and is equal to a deterministic function $\L_\eta(\l)$. Moreover, we will derive a probabilistic formulation of both $\L_\eta(\l)$ and its derivative. 
We begin by expressing the moment generating function of $T_n$ in terms products of the matrices $\Phi_k(\l)$. For ease of notation we introduce the notation
\[
 \Phi_{[m,n]}(\l) = \prod_{k=m}^n \Phi_k(\l), \quad \text{for any } m\leq n.
\]
With this notation, it is easy to see that $E_\w^\pi[ e^{\l T_n} \ind{T_n < \infty} ] =  \pi \Phi_{[0,n-1]}(\l) \one$, where on the right side $\pi = (\pi(1),\pi(2),\ldots,\pi(d))$ is a row vector and $\one$ is a column vector of all 1's. 
To identify the limit of $n^{-1} \log( \pi \Phi_{[0,n-1]}(\l) \one )$, we first need the following Lemma. 

\begin{lem}\label{mudeflem}
If for some $\l \in \R$ and $\w \in \Omega$ there exists a $c>0$ such that $1/c \leq \Phi_k(\l)(i,j) \leq 1/c$ for all $k \in \Z$, $i,j \in [d]$, then 
there exists a sequence of vectors $\{ \mu_n(\l,\w) \}_{n\in\Z}$ such that
\begin{equation}\label{muerror}
 \sup_{\mathbf{0} \neq \pi \geq 0} \left| \frac{\pi \Phi_{[m,n-1]}(\l) }{ \pi \Phi_{[m,n-1]}(\l) \mathbf{1} } - \mu_n(\l,\w) \right\|_1 \leq \frac{2 (1-c^4)^{n-m-1}}{c^4}, \quad \forall n\in\Z, \, m \leq n-1. 
\end{equation}
\end{lem}
\begin{rem}
The vectors $\mu_n(\l,\w)$ are necessarily non-negative with entries summing to 1 and thus can be viewed as probability distributions on $[d]$ that depend on the environment $\w$. For convenience of notation we will often suppress the dependence on $\w$ and just write $\mu_n(\l)$ instead. 
\end{rem}

\begin{cor}\label{mudefcor}
 Let $\eta$ be a measure on environments satisfying Assumptions \ref{asmerg} and \ref{easm}. Then the sequence of vectors $\{\mu_n(\l)\}_{n \in \Z}$ from Lemma \ref{mudeflem} is an ergodic sequence. 
\end{cor}
\begin{proof}
Lemma \ref{PhiUboundlem} implies that the conditions of Lemma \ref{mudeflem} are satisfied for $\l \leq \lcrit$ and $\eta$-a.e.\ environment $\w$ with $c = c_\l$. Thus \eqref{muerror} implies that
\[
 \mu_n(\l) = \lim_{m \ra -\infty} \frac{e_i \Phi_{[m,n-1]}(\l) }{ e_i \Phi_{[m,n-1]}(\l) \mathbf{1} },  \quad \eta\text{ - a.s.},
\]
where the limit doesn't depend on the choice of $i \in [d]$. 
This shows that $\mu_n(\l) = \mu_n(\l,\w)$ is a deterministic function of the environment that commutes with shifts of the environment in the sense that $\mu_n(\l,\w) = \mu_0(\l,\th^n\w)$. Since the environment $\w$ is ergodic by assumption, it follows that $\mu_n(\l)$ is ergodic as well.
\end{proof}

We postpone for the moment the proof of Lemma \ref{mudeflem} and instead show how Corollary \ref{mudefcor} can be used to prove the following statement for the limit of $n^{-1} \log E_\w^\pi[ e^{\l T_n} \ind{T_n < \infty} ]$.  
\begin{lem}\label{qlmgflim}
 For a distribution $\eta$ on environments satisfying Assumptions \ref{asmerg} and \ref{easm}, define
\[
 \L_\eta(\l) =
\begin{cases}
 E_\eta[ \log ( \mu_0(\l) \Phi_0(\l) \mathbf{1}) ] & \l \leq \lcrit \\
 \infty & \l > \lcrit.
\end{cases}
\]
Then for any distribution $\pi$ on $[d]$ for the height of the starting location of the walk ($\pi$ can even be random depending on $\w$),
\begin{equation}\label{qlmgflimeq}
 \lim_{n\ra\infty} \frac{1}{n} \log E_\w^{\pi}\left[ e^{\l T_n} \ind{T_n < \infty} \right] = \L_\eta(\l), \quad \forall \l, \quad \eta\text{ - a.s.}
\end{equation}
\end{lem}
\begin{proof}
First, we claim that it is enough to prove \eqref{qlmgflimeq} holds $\eta$-a.s.\ for any fixed $\l$, and thus for $\eta$-a.e.\ environment the limit holds for all rational $\l$. It will be shown below that the function $\L_\eta(\l)$ is continuous on $(-\infty,\lcrit]$ (this will follow from the fact that $\l \mapsto \mu_0(\l)$ is continuous), and since the left side of \eqref{qlmgflimeq} is a monotone function of $\l$ for every $n$ we can conclude that for $\eta$-a.e.\ environment the limit in \eqref{qlmgflimeq} holds for all $\l$. 

Since $E_\w^\pi[e^{\l T_n} \ind{T_n < \infty} ]$ is finite if and only if $\l \leq \lcrit$, it is enough to consider the case when $\l \leq \lcrit$. For a fixed $\l \leq \lcrit$ note that
\begin{align}
 \log E_\w^\pi\left[ e^{\l T_n} \ind{T_n < \infty} \right]
= \log( \pi \Phi_{[0,n-1]}(\l) \one ) 
&= \log( \pi \Phi_0(\l) \one) + \sum_{k=1}^{n-1} \log\left( \frac{\pi \Phi_{[0,k]}(\l) \one}{ \pi \Phi_{[0,k-1]}(\l) \one} \right) \nonumber \\
&=  \log( \pi \Phi_0(\l) \one) + \sum_{k=1}^{n-1} \log\left( \frac{\pi \Phi_{[0,k-1]}(\l)}{ \pi \Phi_{[0,k-1]}(\l) \one} \Phi_k(\l) \one \right) \nonumber \\
&=: \sum_{k=0}^{n-1} \log( z_k \Phi_k(\l) \one), \label{zsum}
\end{align}
where the last equality is used to define the vectors $z_k$. 
Note that from the formulas for $z_k$ given above it is clear that for $k$ large we should be able to approximate $z_k$ by $\mu_k(\l)$.
Indeed, \eqref{muerror} implies that $\| z_k - \mu_k(\l) \|_1 \leq \frac{2}{c_\l^4} (1-c_\l^4)^{k-1}$.
Now, for any probability vector $\mu$ on $[d]$, Lemma \ref{PhiUboundlem} implies that
\[
 \mu \Phi_k(\l) \mathbf{1} = \sum_{i,j} \mu(i) \Phi_k(\l)(i,j) \geq c_\l d.
\]
Then since $|\log(x) - \log(y)| \leq \frac{|x-y|}{a}$ for any $x,y \geq a$, it follows that
\begin{align}
 | \log( \mu_k(\l) \Phi_k(\l) \mathbf{1} ) - \log( z_k \Phi_k(\l) \mathbf{1} ) |
&\leq \frac{1}{c_\l d} | \mu_k(\l) \Phi_k(\l) \mathbf{1} - z_k \Phi_k(\l) \mathbf{1} | \nonumber \\
&\leq \frac{1}{c_\l d} \|\mu_k(\l) - z_k \|_1 \| \Phi_k(\l) \| 
\leq \frac{ 2}{c_\l^6} (1-c_\l^4)^{k-1}, \label{logbound}
\end{align}
where in the last inequality we used $\| \Phi_k(\l) \| \leq d/c_\l$ which follows from the upper bound on the entries of $\Phi_k(\l)$ in \eqref{PhiUbound}.
Combining \eqref{zsum} and \eqref{logbound} we see that
\begin{equation}\label{eq:qlmgfUbound}
 \left| \frac{1}{n} \log E_\w^\pi[ e^{\l T_n} \ind{T_n < \infty} ] - \frac{1}{n} \sum_{k=0}^{n-1} \log(\mu_k(\l) \Phi_k(\l) \mathbf{1} ) \right|
\leq \frac{1}{n} \sum_{k=0}^{n-1} \frac{ 2}{c_\l^6} (1-c_\l^4)^{k-1}
\leq \frac{2(1-c_\l^4)^{-1}}{c_\l^{10} n},
\end{equation}
and since the expression on the right vanishes as $n\ra\infty$ it is enough to evaluate the limit of the second sum on the left side. However, since the environment $\w$ is ergodic it follows that $\{ \mu_k(\l) \Phi_k(\l) \mathbf{1} \}_{k\in\Z}$ is ergodic as well since both $\mu_k(\l)$ and $\Phi_k(\l)$ are functions of the shifted environment $\th^k \w$. Thus Birkhoff's ergodic theorem implies that
\[
 \lim_{n\ra\infty} \frac{1}{n} \sum_{k=0}^{n-1} \log(\mu_k(\l) \Phi_k(\l) \mathbf{1} ) = E_\eta[ \log(\mu_0(\l) \Phi_0(\l) \mathbf{1} ) ], \quad \eta\text{ - a.s.}
\]
Combining this with \eqref{eq:qlmgfUbound} finishis the proof of the lemma. 
\end{proof}

We now return to the proof of the existence of the vectors $\mu_n(\l)$ and the associated error bounds as stated in Lemma \ref{mudeflem}. 

\begin{proof}[Proof of Lemma \ref{mudeflem}]
The key to the proof of Lemma \ref{mudeflem} is the following Lemma from \cite{bgSRT}.

\begin{lem}[Lemma 9 in \cite{bgSRT}]\label{pr}
 Let $G_n$, $n=1,2,\ldots$ be a sequence of $d\times d$ matrices with all positive entries, and for any $r \geq 2$ let
\[
 \rho_r = \min_{i,j,k} \frac{ G_r(i,j)G_{r-1}(j,k) }{ \sum_{\ell} G_r(i,\ell)G_{r-1}(\ell,k)}.
\]
 If $\sum_{r=2}^\infty \rho_r = \infty$ then there exists a vector $\vec{v} = (v(1),v(2),\ldots,v(d))$ with strictly positive entries adding to 1 such that for any $n\geq 2$
\[
 G_n G_{n-1} \cdots G_1 = D_n
\left\{
\left(
\begin{array}{cccc}
 | & | & & | \\
v(1) & v(2) & \cdots & v(d) \\
 | & | & & |
\end{array}
\right)
+ \e_{n}
\right\},
\]
where $D_n$ is a $d\times d$ positive diagonal matrix and $\e_n$ is a $d\times d$ matrix with $\|\e_n\| \leq \prod_{r=2}^n(1- d \rho_r)$.
\end{lem}

For fixed $\w \in \Omega$, $\l \in \R$, and $c>0$ satisfying the assumptions of Lemma \ref{mudeflem} and for $n \in \Z$ fixed, we will apply Lemma \ref{pr} with $G_k = \Phi_{n-k}(\l)$. 
The representation of the product $G_k G_{k-1} \cdots G_1$ from Lemma \ref{pr} implies that
\[
 \frac{e_i G_k G_{k-1} \cdots G_1}{e_i G_k G_{k-1} \cdots G_1 \one} = \frac{ \vec{v} + e_i \e_k }{ 1 + e_i \e_k \one }.
\]
The uniform upper and lower bounds on the entries of $\Phi_{n-k}(\l) = G_k$ imply that $\rho_r \geq c^4/d$ for all $r\geq 2$, and so the matrix $\e_k$ has norm $\| \e_k \| \leq (1-c^4)^{k-1}$ for $k\geq 2$.
Thus, it follows that for any $i \in [d]$ and $k\geq 2$,
\begin{align*}
\left\| \frac{e_i G_k G_{k-1} \cdots G_1}{e_i G_k G_{k-1} \cdots G_1 \one} - \vec{v} \right\|_1
&
\leq \|\vec{v} \|_1 \left| \frac{1}{1+e_i \e_k \one} - 1 \right| + \frac{ \| e_i \e_k \|_1 }{ |1+e_i \e_k \one |} 
\leq \frac{2 \| \e_k \|}{1-\|\e_k\|} \\
&\leq \frac{2 (1-c^4)^{k-1}}{1-(1-c^4)^{k-1}} \leq \frac{2(1-c^4)^{k-1}}{c^4}.
\end{align*}
Since $\| e_i G_1/(e_i G_1 \one) \|_1 = \|\vec{v} \|_1= 1$ and $2/c^4 > 2$, the above error bound also holds when $k=1$.
Finally, since for any non-negative vector $\pi \neq \mathbf{0}$ and any non-negative matrix $A$, the vector $\pi A /(\pi A \mathbf{1})$ is a convex combination of the vectors $\{ e_i A / (e_i A \mathbf{1}) \}_{i\in[d]}$, it follows that
\[
 \sup_{\mathbf{0} \neq \pi \geq 0} \left\| \frac{\pi G_k G_{k-1} \cdots G_1}{\pi G_k G_{k-1} \cdots G_1 \one} - \vec{v} \right\|_1
\leq \frac{2(1-c^4)^{k-1}}{c^4}, \quad \forall k\geq 1.
\]
This completes the proof of Lemma \ref{mudeflem} with $\mu_n(\l,\w) = \vec{v}$. 
\end{proof}

\subsection{Differentiability of \texorpdfstring{$\L_\eta(\l)$}{the quenched log moment generating function}}\label{diffsec}

For any environment $\w$ and $i \in [d]$, let $$\L_{\w,i,n}(\l) = \frac{1}{n} \log E_\w^{(0,i)}[ e^{\l T_n} \ind{T_n < \infty} ].$$
For any fixed $i \in [d]$, $n \geq 1$ and $\w$, the function $\L_{\w,i,n}(\l)$ is strictly convex and differentiable (analytic even) on $(-\infty,\lcrit)$.
Since $\L_{\w,i,n}(\l) \ra \L_\eta(\l)$ as $n\ra\infty$, it follows that $\L_\eta(\l)$ is a convex function on $(-\infty,\lcrit)$ but a priori we cannot conclude that $\L_\eta(\l)$ is differentiable on $(-\infty,\lcrit)$.
If we can show that $\L_{\w,i,n}'(\l)$ converges uniformly on compact intervals then it will follow that $\L_\eta(\l)$ is differentiable and that $\L_\eta'(\l) = \lim_{n\ra\infty} \L'_{\w,i,n}(\l)$ \cite[Theorem 7.17]{rPOMA}.

For any $k \geq 1$, define $\tau_k = T_k-T_{k-1}$ when $T_{k-1} < \infty$ and $\tau_k = \infty$ otherwise, so that $T_n = \sum_{k=1}^n \tau_k$.
Then, it's easy to see that for $\l < \lcrit$,
\[
 \L_{\w,i,n}'(\l) =  \frac{1}{n} \left( \frac{ E_{\w}^{(0,i)}[ T_n e^{\l T_n} \ind{T_n < \infty} ] }{ E_\w^{(0,i)}[ e^{\l T_n}  \ind{T_n < \infty}] } \right)
= \frac{1}{n} \sum_{k=1}^n \frac{ E_{\w}^{(0,i)}[ \tau_k e^{\l T_n} \ind{T_n < \infty} ] }{ E_\w^{(0,i)}[ e^{\l T_n}  \ind{T_n < \infty}] }.
\]
Let $\Phi_k'(\l)$ be the term-by-term derivative of the matrix $\Phi_k(\l)$. That is,
\[
 \Phi_k'(\l)(i,j) = E_\w^{(k,i)}\left[ T_{k+1} e^{\l T_{k+1}} \ind{T_{k+1}<\infty,\, Y_{T_{k+1}} = j} \right].
\]
Then, with this notation we have that
\begin{equation}\label{dproduct}
 \L_{\w,i,n}'(\l)
=  \frac{1}{n} \sum_{k=1}^n \frac{ e_i \Phi_{[0,k-2]}(\l) \Phi_{k-1}'(\l) \Phi_{[k,n-1]}(\l) \mathbf{1} } { e_i \Phi_{[0,n-1]}(\l) \mathbf{1} }.
\end{equation}

To approximate \eqref{dproduct} we want to approximate $\Phi_{[k,n-1]}(\l) \mathbf{1}$ in both the numerator and denominator. These terms will grow or decrease exponentially, but if we normalize them they converge. To show this we need the following lemma.
\begin{lem}\label{nudeflem}
If for some $\l \in \R$ and $\w \in \Omega$ there exists a $c>0$ such that $1/c \leq \Phi_k(\l)(i,j) \leq 1/c$ for all $k \in \Z$, $i,j \in [d]$, 
then there exists a sequence of non-negative vectors $\{ \nu_k(\l) \}_{k\in\Z}$ such that
\begin{equation}\label{nuerror}
 \left\|  \frac{\Phi_{[k,n-1]}(\l) \mathbf{1} }{ \mathbf{1}^t \Phi_{[k,n-1]}(\l) \mathbf{1} } - \nu_k(\l) \right\|_\infty \leq \frac{2  (1-c^4)^{n-k-1} }{c^4}, \quad \forall k < n. 
\end{equation}
Moreover, if $\eta$ satisfies assumptions \ref{asmerg} and \ref{easm} then $\nu_k(\l)$ exists with probability 1 for all $\l\leq \lcrit$ and the sequence $\{\nu_k(\l)\}_{k\in\Z}$ is ergodic.
\end{lem}
\begin{proof}
For any matrix $A$, let $A^t$ denote the transpose of $A$. Then, 
\[
 \left( \frac{\Phi_{[k,n-1]}(\l) \mathbf{1} }{ \mathbf{1}^t \Phi_{[k,n-1]}(\l) \mathbf{1} } \right)^t
= \frac{ \one^t \Phi_{n-1}(\l)^t \Phi_{n-2}(\l)^t \cdots \Phi_{k+1}(\l)^t \Phi_k(\l)^t }{ \one^t \Phi_{n-1}(\l)^t \Phi_{n-2}(\l)^t \cdots \Phi_{k+1}(\l)^t \Phi_k(\l)^t \one }. 
\]
Then similarly to the proof of Lemma \ref{mudeflem}, the proof of \eqref{nuerror} follows by applying Lemma \ref{pr} with $G_j = \Phi_{k+j-1}(\l)^t$. 
The final claims in the statement of Lemma \ref{nudeflem} follow as in the proof of Corollary \ref{mudefcor}. 
\end{proof}

We'll also need a uniform upper bound on $\| \Phi_k'(\l) \|$ for $\l < \lcrit$.
\begin{lem}\label{dphiUbound}
For any $\l < \lcrit$, there exists a constant $d_\l < \infty$ such that
\[
 \eta\left( \| \Phi_0'(\l) \| \leq d_\l \right) = 1.
\]
\end{lem}
\begin{proof}
Since a uniform bound on the entries of the matrix $\Phi_0'(\l)$ implies a uniform bound on the matrix norm it is enough to show a uniform bound on the entries of $\Phi_0'(\l)$. 
Since for any $i,j \in[d]$,  $\Phi_0(\l)(i,j)$ is a convex function of $\l$ we have that
\begin{equation}\label{secbound}
 \Phi_0'(\l)(i,j) \leq \frac{ \Phi_0(\lcrit)(i,j) - \Phi_0(\l)(i,j) }{\lcrit - \l} \leq \frac{(1/c_{\lcrit}) - c_\l}{\lcrit - \l},
\end{equation}
where the last inequality follows from Lemma \ref{PhiUboundlem}.
\end{proof}

Having laid the necessary groundwork, we are ready to prove that $\L_\eta(\l)$ is differentiable and give a formula for the derivative. 
\begin{lem}\label{derivform}
If the distribution $\eta$ on environments satisties Assumptions \ref{asmerg} and \ref{easm}, then
$\L_\eta(\l)$ is continuously differentiable on $(-\infty, \lcrit)$ and
\[
 \L_\eta'(\l) = E_\eta \left[ \frac{ \mu_{0}(\l) \Phi_{0}'(\l) \nu_1(\l) }{ \mu_0(\l) \Phi_0(\l) \nu_1(\l) } \right], \quad \forall \l < \lcrit.
\]
\end{lem}
\begin{proof}
As mentioned above, it is enough to show that
\begin{equation}\label{dLwinconv}
 \lim_{n\ra\infty} \L_{\w,i,n}'(\l) = E_\eta \left[ \frac{ \mu_{0}(\l) \Phi_{0}'(\l) \nu_1(\l) }{\mu_0(\l) \Phi_0(\l) \nu_1(\l) } \right], \quad \eta\text{ - a.s.},
\end{equation}
and that the convergence is uniform in $\l$ on compact subsets of $(-\infty,\lcrit)$.
To this end, we first note that since $\L_{\w,i,n}'(\l)$ is continuous and non-decreasing in $\l$, uniform convergence on compact subsets will follow from pointwise convergence if we can show that the proposed limit $E_\eta \left[ \frac{ \mu_{0}(\l) \Phi_{0}'(\l) \nu_1(\l) }{ \mu_0(\l) \Phi_0(\l) \nu_1(\l) } \right]$ is also continuous in $\l$. Since we can uniformly bound each of the terms inside the expectation, it is enough to show that each of these terms is continuous in $\l$. It is obvious from their definitions as quenched expectations that $\Phi_0(\l)$ and $\Phi_0'(\l)$ are continuous in $\l$, but we need to prove that $\mu_0(\l)$ and $\nu_1(\l)$ are continuous in $\l$. To show that $\mu_0(\l)$ is continuous in $\l$, first note that for any $\l,\l' \leq \lcrit$ and any $n \geq 1$, the error bounds in \eqref{muerror} imply that
\begin{align*}
 \| \mu_0(\l) - \mu_0(\l') \|_1 \leq \frac{2}{c_\l^4}(1-c_\l^4)^{n-1} + \frac{2}{c_{\l'}^4}(1-c_{\l'}^4)^{n-1} + \left\| \frac{ e_i \Phi_{[-n,-1]}(\l) }{ e_i \Phi_{[-n,-1]}(\l) \mathbf{1} } - \frac{ e_i \Phi_{[-n,-1]}(\l') }{ e_i \Phi_{[-n,-1]}(\l') \mathbf{1} } \right\|_1.
\end{align*}
Since the proof of Lemma \ref{PhiUboundlem} shows that the constants $c_\l$ are continuous in $\l$, we obtain that 
\[
 \lim_{\l' \ra \l} \|  \mu_0(\l) - \mu_0(\l') \|_1 \leq \frac{4}{c_\l^4}(1-c_\l^4)^{n-1}, \quad \forall \l \leq \lcrit. 
\]
Since this holds for any $n\geq 1$, taking $n\ra\infty$ shows that $\l \mapsto \mu_0(\l)$ is
continuous, $\eta$-a.s. A similar argument shows that $\nu_1(\l)$ is continuous, $\eta$-a.s., and thus that the right side of \eqref{dLwinconv} is continuous for $\l <\lcrit$.

It remains to prove the pointwise convergence in \eqref{dLwinconv}. To this end, note that the terms in the sum on the right of \eqref{dproduct} can be re-written (for $2\leq k \leq n-1$) as
\begin{equation} \label{altrep}
 \frac{ e_i \Phi_{[0,k-2]}(\l) \Phi_{k-1}'(\l) \Phi_{[k,n-1]}(\l) \mathbf{1} }{ e_i  \Phi_{[0,n-1]}(\l) \mathbf{1} }
=
\frac{ \left( \frac{ e_i \Phi_{[0,k-2]}(\l) }{e_i \Phi_{[0,k-2]}(\l)\mathbf{1} } \right) \Phi_{k-1}'(\l) \left( \frac{ \Phi_{[k,n-1]}(\l) \mathbf{1}}{ \mathbf{1}^t \Phi_{[k,n-1]}(\l) \mathbf{1}} \right) }
{ \left( \frac{ e_i \Phi_{[0,k-2]}(\l) }{e_i \Phi_{[0,k-2]}(\l)\mathbf{1} } \right) \Phi_{k-1}(\l) \left( \frac{ \Phi_{[k,n-1]}(\l) \mathbf{1}}{ \mathbf{1}^t \Phi_{[k,n-1]}(\l) \mathbf{1}} \right) }.
\end{equation}
We would like to approximate the numerator of the fraction on the right by $\mu_{k-1}(\l) \Phi_{k-1}'(\l) \nu_k(\l)$ and the denominator by $\mu_{k-1}(\l) \Phi_{k-1}(\l) \nu_k(\l)$. Equations \eqref{muerror}, \eqref{nuerror} and Lemma \ref{dphiUbound} imply that
there exists a constant $C$ depending on $\l$  such that
\begin{align}
 & \left| \left( \frac{ e_i \Phi_{[0,k-2]}(\l) }{e_i \Phi_{[0,k-2]}(\l)\mathbf{1} } \right) \Phi_{k-1}'(\l) \left( \frac{ \Phi_{[k,n-1]}(\l) \mathbf{1}}{ \mathbf{1}^t \Phi_{[k,n-1]}(\l) \mathbf{1}} \right) - \mu_{k-1}(\l) \Phi_{k-1}'(\l) \nu_k(\l) \right| \nonumber \\
&\quad \leq \left\| \frac{ e_i \Phi_{[0,k-2]}(\l) }{e_i \Phi_{[0,k-2]}(\l)\mathbf{1} } - \mu_{k-1}(\l) \right\|_1 \| \Phi_{k-1}'(\l) \| \|\nu_k(\l)\|_\infty \nonumber \\
&\qquad +  \left\| \frac{ e_i \Phi_{[0,k-2]}(\l) }{e_i \Phi_{[0,k-2]}(\l)\mathbf{1} } \right\|_1 \| \Phi_{k-1}'(\l) \| \left\|  \frac{ \Phi_{[k,n-1]}(\l) \mathbf{1}}{ \mathbf{1}^t \Phi_{[k,n-1]}(\l) \mathbf{1}} - \nu_k(\l) \right\|_\infty \nonumber \\
&\quad \leq C (1-c_\l^4)^{k\wedge (n-k)}, \label{numerror}
\end{align}
and similarly there is a constant $C'$ (also depending on $\l$) such that
\begin{align}
& \left|
 \left( \frac{ e_i \Phi_{[0,k-2]}(\l) }{e_i \Phi_{[0,k-2]}(\l)\mathbf{1} } \right) \Phi_{k-1}(\l) \left( \frac{ \Phi_{[k,n-1]}(\l) \mathbf{1}}{ \mathbf{1}^t \Phi_{[k,n-1]}(\l) \mathbf{1}} \right)
 - \mu_k(\l) \Phi_{k-1}(\l) \nu_k(\l) \right|
\leq C' (1-c_\l^4)^{k \wedge (n-k)}. \label{denerror}
\end{align}

The error bounds in \eqref{numerror} and \eqref{denerror} allow us to approximate the numerator and denominator from \eqref{altrep} separately, but in order to approximate the ratio we also need to obtain an upper bound on the numerator terms and a lower bound on the denominator terms. 
Lemma \eqref{dphiUbound} gives a uniform upper bound on the numerator terms, and if we denote the denominator by $\mu \Phi_{k-1}(\l) \nu$ then since the vectors $\mu$ and $\nu$ are both non-negative with entries summing to 1, Lemma \ref{PhiUboundlem} implies that the denominator of the right side of \eqref{altrep} is bounded below by
\[
 \sum_{i,j \in [d]} \mu(i) \Phi_{k-1}(\l)(i,j) \nu(j) \geq \sum_{i,j \in [d]} \mu(i) c_\l \nu(j)
= c_\l.
\]
Thus, with these upper bounds on the numerator and lower bounds on the denominator we can combine \eqref{altrep}, \eqref{numerror} and \eqref{denerror} to conclude that for some constant $C''<\infty$ depending on $\l$ that
\[
 \left | \frac{ e_i \Phi_{[0,k-2]}(\l) \Phi_{k-1}'(\l) \Phi_{[k,n-1]}(\l) \mathbf{1} }{ e_i  \Phi_{[0,n-1]}(\l) \mathbf{1} }
 - \frac{ \mu_{k-1}(\l) \Phi_{k-1}'(\l) \nu_k(\l) }{ \mu_{k-1}(\l) \Phi_{k-1}(\l) \nu_k(\l) } \right| \leq C'' (1-c_\l^4)^{k\wedge (n-k)}.
\]
This is enough to imply that
\begin{align*}
\lim_{n\ra\infty} \L_{\w,i,n}'(\l) 
& = \lim_{n\ra\infty} \frac{1}{n} \sum_{k=1}^n \frac{ \mu_{k-1}(\l) \Phi_{k-1}'(\l) \nu_k(\l) }{  \mu_{k-1}(\l) \Phi_{k-1}(\l) \nu_k(\l) }  \\
&= E_\eta\left[ \frac{ \mu_{0}(\l) \Phi_{0}'(\l) \nu_1(\l) }{  \mu_{0}(\l) \Phi_{0}(\l) \nu_1(\l) } \right], \quad \eta\text{ - a.s.},
\end{align*}
where the last equality follows from Birkhoff's ergodic theorem.
\end{proof}

\subsection{Truncated log moment generating functions}

For certain parts of the proofs of the main results, it will be important to have modified versions of the previous results in this section when the moment generating functions $\Phi_k(\l)$ are replaced by the truncated versions $\Phi_{k,M}(\l)$ as defined in \eqref{PhikMdef}. 
In the following we will use the notation $\Phi_{[m,n],M}(\l) = \Phi_{m,M}(\l) \Phi_{m+1,M}(\l) \cdots \Phi_{n,M}(\l)$ for any $m\leq n$. 
First, we prove corresponding results for truncated versions of $\mu_n(\l)$ and $\nu_n(\l)$ exist.

\begin{lem}\label{muMdeflem}
For every $\w \in \Omega_\k$, $M \geq N_\k$, $\l \in \R$, and $n \in \Z$, there exist vectors $ \mu_{n,M}(\l)$ and $\nu_{n,M}(\l)$ such that
\[
 \mu_{n,M}(\l) = \lim_{m \ra -\infty}  \frac{e_i \Phi_{[m,n-1],M}(\l) }{ e_i \Phi_{[m,n-1],M}(\l) \mathbf{1} },  
\quad\text{and}\quad 
 \nu_{n,M}(\l) = \lim_{m \ra\infty} \frac{ \Phi_{[n,m],M}(\l) \mathbf{1} }{ \mathbf{1}^t \Phi_{[n,m],M}(\l) \mathbf{1} },
\]
where the limit in the definition of $\mu_{n,M}(\l)$ doesn't depend on $i \in [d]$. 
If in addition $\eta$ satisfies Assumptions \ref{asmerg} and \ref{easm} then the sequences $\mu_{n,M}(\l)$ and $\nu_{n,M}(\l)$ are ergodic, and for any $\l \leq \lcrit(\eta)$ the error bounds 
\begin{equation}\label{muMerror}
 \sup_{\mathbf{0} \neq \pi \geq 0} \left| \frac{\pi \Phi_{[m,n-1],M}(\l) }{ \pi \Phi_{[m,n-1],M}(\l) \mathbf{1} } - \mu_{n,M}(\l) \right\|_1 \leq \frac{2 (1-c_\l^4)^{n-m-1}}{c_\l^4}, \quad \forall m < n, 
\end{equation}
and 
\begin{equation}\label{nuMerror}
 \left| \frac{ \Phi_{[n,m],M}(\l) \one }{ \one^t \Phi_{[n,m],M}(\l) \mathbf{1} } - \nu_{n,M}(\l) \right\|_1 \leq \frac{2 (1-c_\l^4)^{m-n}}{c_\l^4}, \quad \forall n\leq m,
\end{equation}
hold for $\eta$-a.e.\ environment $\w$, where the $c_\l$ are the constants from Lemma \ref{PhiUboundlem}. Moreover, 
\[
 \lim_{M\ra\infty} \mu_{n,M}(\l) = \mu_n(\l), \text{ and } \, \lim_{M \ra\infty} \nu_{n,M}(\l) = \nu_n(\l), \quad \eta\text{ - a.s.}, \quad \forall \l \leq \lcrit(\eta).
\]
\end{lem}
\begin{proof}
The key to the proofs of Lemmas \ref{mudeflem} and \ref{nudeflem} were the uniform upper and lower bounds on the entries of $\Phi_k(\l)$ from Lemma \ref{PhiUboundlem}. 
However, as noted in Remark \ref{rem:PhikMUbounds} above, for $M \geq N_\k$ and $\l \leq \lcrit(\eta)$ the same uniform upper and lower bounds holds for the entries of $\Phi_{k,M}(\l)$ and $\Phi_k(\l)$. Moreover, there are uniform upper and lower bounds on the entries of $\Phi_{k,M}(\l)$ when $\l > \lcrit(\eta)$ as well since 
\[
 e^\l \k \leq e^\l P_\w^{(0,i)}( T_1 \leq M, \, Y_{T_1} = j ) \leq \Phi_{0,M}(\l)(i,j) \leq e^{\l M}, \quad \forall \w \in \Omega_\k, \, \l > 0, \, M > N_\k. 
\]
This shows that the limits defining $\mu_{n,M}(\l)$ and $\nu_{n,M}(\l)$ exist. Moreover, since the uniform and lower bounds are the same for $\Phi_k(\l)$ and $\Phi_{k,M}(\l)$ when $\l \leq \lcrit(\eta)$ the error bound in \eqref{muMerror} is the same as the one in \eqref{muerror}.

Finally, we will show that $\mu_{n,M}(\l) \ra \mu_n(\l)$ (the proof that $\nu_{n,M}(\l) \ra \nu_n(\l)$ is similar).
Since the entries of $\mu_{n,M}(\l)$ are bounded, let $M_k \ra \infty$ be a subsequence where the limit exists and denote the limit by $\mu_n^*(\l)$.
By Lemma \ref{mudeflem}, for any $\e>0$ we can be choose $m = m(n,\e,\l) < n$ so that any probability distribution $\pi$ on $[d]$ satisfies
\begin{equation}\label{munl1e}
 \left\| \frac{\pi \Phi_{[m,n-1]}(\l) }{ \pi \Phi_{[m,n-1]}(\l) \mathbf{1} } - \mu_n(\l) \right\|_1 < \e.
\end{equation}
Now, for this $m$ fixed there exists a further subsequence $M_k'$ of $M_k$ such that $\lim_{k\ra\infty} \mu_{m,M_k'}(\l)$ also exists, and we will denote this limit by $\mu_m^*(\l)$. The definition of $\mu_{k,M}(\l)$ ensures that
\[
 \frac{ \mu_{m,M}(\l) \Phi_{[m,n-1],M}(\l) }{ \mu_{m,M}(\l) \Phi_{[m,n-1],M}(\l) \mathbf{1} } = \mu_{n,M}(\l), 
\]
and by taking limits of this equality along the subsequence $M_k'$ we obtain that
\[
 \frac{ \mu_m^*(\l) \Phi_{[m,n-1]}(\l) }{ \mu_m^*(\l) \Phi_{[m,n-1]}(\l) \mathbf{1} } = \mu_n^*(\l) .
\]
Finally, applying \eqref{munl1e} with $\pi = \mu_m^*(\l)$ we can conclude that $\|\mu_n^*(\l) - \mu_n(\l)\|_1 < \e$. Since $\e>0$ was arbitrary we conclude that $\mu_n^*(\l) = \mu_n(\l)$ and so any subsequential limit of $\mu_{n,M}(\l)$ must equal $\mu_n(\l)$. 
\end{proof}

Next, we prove a truncated version of Lemmas \ref{qlmgflim} and \ref{derivform}. 
\begin{lem}\label{tqlmgflim}
 For any distribution $\pi$ (even depending on $\w$) for the height of the initial location of the random walk,
\[
 \lim_{n\ra\infty} \frac{1}{n} \log E_\w^{\pi} \left[ e^{\l T_n} \ind{\tau_k \leq M, \, k=1,2,\ldots n} \right]
= E_\eta\left[ \log \left( \mu_{0,M}(\l) \Phi_{0,M}(\l) \mathbf{1} \right) \right]
=: \L_{\eta,M}(\l), \quad \eta\text{ - a.s}.
\]
$\L_{\eta,M}(\l)$ is convex in $\l$ and continuously differentiable for all $\l \in \R$ with
\[
 \L_{\eta,M}'(\l)
= E_\eta \left[ \frac{ \mu_{0,M}(\l) \Phi_{0,M}'(\l) \nu_{1,M}(\l) }{ \mu_{0,M}(\l) \Phi_{0,M}(\l) \nu_{1,M}(\l) } \right].
\]
Moreover, $\lim_{M\ra\infty} \L_{\eta,M}(\l) = \L_\eta(\l)$ for all $\l \in \R$ and $\lim_{M\ra\infty} \L_{\eta,M}'(\l) = \L_\eta'(\l)$ for all $\l < \lcrit$.
\end{lem}
\begin{proof}
Since we can represent the expectation as a matrix product by
\[
  E_\w^{\pi} \left[ e^{\l T_n} \ind{\tau_k \leq M, \, k=1,2,\ldots n} \right] = \pi \Phi_{0,M}(\l)\Phi_{1,M}(\l) \cdots \Phi_{n-1,M}(\l) \mathbf{1},
\]
the proof that the limit exists and the formula for the limit is the same as in the proof of Lemma \ref{qlmgflim} and depends only on the uniform upper and lower bounds on the entries of $\Phi_{k,M}(\l)$. 
Similarly, the proof of the formula for $\L_{\eta,M}'(\l)$ is essentially unchanged from the proof of Lemma \ref{derivform}. 

To show that $\L_{\eta,M}(\l) \ra \L_\eta(\l)$, first note that $\Phi_{0,M}(\l) \ra \Phi_0(\l)$ as $M\ra\infty$ by the monotone convergence theorem. If $\l \leq \lcrit$, then Lemma \ref{muMdeflem} and the bounded convergence theorem imply that 
\[
\lim_{M\ra\infty} E_\eta\left[ \log( \mu_{0,M}(\l) \Phi_{0,M}(\l) \mathbf{1} ) \right]
 =  E_\eta\left[ \log( \mu_{0}(\l) \Phi_{0}(\l) \mathbf{1} ) \right]
,\quad \forall \l \leq \lcrit.
\]
For $\l > \lcrit$ we need to show that $\lim_{M\ra\infty} \L_{\eta,M}(\l) = \infty$. To this end, note that
\[
 \L_{\eta,M}(\l) \geq E_\eta\left[ \min_{i \in [d]} \log\left( \sum_{j\in[d]} \Phi_{0,M}(\l)(i,j) \right) \right].
\]
Then, since $\Phi_{0,M}(\l) \nearrow \Phi_0(\l)$ as $M \nearrow \infty$ and $\sum_{j\in[d]} \Phi_0(\l)(i,j) = \infty$ for any $i \in [d]$ when $\l > \lcrit$, the monotone convergence theorem implies that $\lim_{M\ra\infty} \L_{\eta,M}(\l) = \infty$.

To prove that $ \L_{\eta,M}'(\l) \ra \L_\eta'(\l)$ for $\l < \lcrit$, first note that $\mu_{0,M}(\l) \ra \mu_0(\l)$, $\nu_{1,M}(\l) \ra \nu_1(\l)$, $\Phi_{0,M}(\l) \ra\ \Phi_0(\l)$ and $\Phi_{0,M}'(\l) \ra \Phi_0'(\l)$ as $M \ra\infty$ for any $\l < \lcrit$.
The uniform bounds in \eqref{eq:PhikMUbounds} and the proof of Lemma \ref{dphiUbound} give uniform upper bounds on the entries of $\Phi_{0,M}'(\l)$ that do not depend on $M$. Combining this with \eqref{eq:PhikMUbounds} and the fact that $\mu_{0,M}(\l)$ and $\nu_{1,M}(\l)$ are non-negative with entries summing to $1$, we conclude by the bounded convergence theorem that
\[
 \lim_{M\ra\infty} E_\eta \left[ \frac{ \mu_{0,M}(\l) \Phi_{0,M}'(\l) \nu_{1,M}(\l) }{ \mu_{0,M}(\l) \Phi_{0,M}(\l) \nu_{1,M}(\l) } \right] =
E_\eta \left[ \frac{ \mu_{0}(\l) \Phi_{0}'(\l) \nu_{1}(\l) }{ \mu_{0}(\l) \Phi_{0}(\l) \nu_{1}(\l) } \right].
\]
\end{proof}

\section{Proof of the quenched LDP for hitting times}\label{sec:qldpTn}

Having proved the necessary facts about the quenched log moment generating function $\L_\eta(\l)$, we will now give the details of the proof of the quenched LDP for hitting times as stated in Theorem \ref{QLDPTn}.
We will begin by first collecting a few necessary facts about the rate function $J_\eta$ (recall that $J_\eta$ was defined in \eqref{Jetadef} as the Legendre dual of $\L_\eta$).
\begin{lem}\label{Jetaprop}
Let $t_0 = t_0(\eta)$ and $\tcrit = \tcrit(\eta)$ be defined by 
\begin{equation}\label{t0def}
 t_0 = \lim_{\l \ra 0^-} \L_\eta'(\l) \quad\text{and}\quad \tcrit = \lim_{\l \ra \lcrit^-} \L_\eta'(\l).
\end{equation}
Then $J_\eta$ is finite, convex and continuous on $[1,\infty)$, decreasing on $[1,t_0]$ and non-decreasing on $[t_0,\infty)$. Moreover,
\begin{equation}\label{Jetadef2}
 J_\eta(t) =
\begin{cases}
 \sup_{\l \leq 0} (\l t - \L_\eta(\l)) &  \text{ if } t \in [1,t_0] \\
 \sup_{\l \geq 0} (\l t - \L_\eta(\l)) &   \text{ if } t \in [t_0,\tcrit] \\
 \lcrit t - \L_\eta(\lcrit) & \text{ if } t \geq \tcrit.
\end{cases}
\end{equation}
\end{lem}
\begin{rem}
 Note that $t_0 = \tcrit$ if $\lcrit = 0$ and that $t_0 = \L'_\eta(0) < \ts$ if $\lcrit > 0$. 
\end{rem}
\begin{proof}
The main thing that needs to be proved is that $\lim_{\l \ra -\infty} \L'_\eta(\l) = 1$. To this end, note that Assumption \ref{easm} implies that 
\[
 \k^n e^{\l n} \leq P_\w^\pi(X_n = n) e^{\l n} \leq E_\w^\pi[ e^{\l T_n} \ind{T_n < \infty} ] \leq e^{\l n}, \quad \forall \l \leq 0. 
\]
Thus, it follows that
\begin{equation}\label{Letalbub}
 \l + \log(\k) \leq \L_\eta(\l) \leq \l, \quad \forall \l \leq 0,
\end{equation}
and since $\L_\eta(\l)$ is convex and differentiable this implies that $\lim_{\l \ra -\infty} \L_\eta'(\l) = 1$. 
The conclusions of the Lemma then follow easily from the fact that $J_\eta(t)$ is the Legendre transform of $\L_\eta(\l)$. 
Indeed, since $\L_\eta$ is continuously differentiable, for any $t \in (1,\tcrit)$ there exists a $\l_t < \lcrit$ such that $\L_\eta'(\l_t) = t$. Note that this choice of $\l_t$ ensures that $J_\eta(t) = \l_t t - \L_\eta(\l_t)$. From this, it is straightforward to prove the stated properties of $J_\eta$. 
\end{proof}

The next Lemma shows that the parameter $t_0$ defined in \eqref{t0def} also has an important probabilistic meaning. 
\begin{lem}\label{lem:t0}
 If the random walk is recurrent or transient to the right, then for any initial distribution $\pi$ for the starting height of the random walk
\[
 \lim_{n\ra\infty} T_n/n = t_0, \quad \P^\pi\text{ - a.s.}
\]
\end{lem}
\begin{rem}
 In light of the law of large numbers for hitting times in \eqref{llnTn} we can conclude that $t_0 = 1/\vp$, where $\vp$ is the limiting speed for the RWRE. 
\end{rem}

\begin{proof}
First, we claim that the formula for $\L_\eta'(\l)$ in Lemma \ref{derivform} implies that
\begin{equation}\label{t0form}
 t_0 = \lim_{\l \ra 0-} E_\eta\left[ \frac{\mu_0(\l) \Phi_0'(\l) \nu_1(\l) }{\mu_0(\l) \Phi_0(\l) \nu_1(\l) } \right]
= E_\eta\left[ \frac{\mu_0(0) \Phi_0'(0) \nu_1(0) }{\mu_0(0) \Phi_0(0) \nu_1(0) } \right].
\end{equation}
If $E_\eta[ \| \Phi_0'(0) \| ] < \infty$ then this follows from the dominated convergence theorem since
\[
 \frac{\mu_0(\l) \Phi_0'(\l) \nu_1(\l) }{\mu_0(\l) \Phi_0(\l) \nu_1(\l) } \leq \frac{\| \Phi_0'(\l) \|}{ c_\l } \leq \frac{\|\Phi_0'(0)\|}{c_{-1}}, \quad \forall \l \in [-1,0]. 
\]
On the other hand, it can be shown that the uniform bounds on the entries of $\Phi_k(\l)$ in \eqref{PhiUbound} imply that all of the entries of $\mu_0(\l)$ and $\nu_1(\l)$ are in $[c_\l^2/d, 1/(c_\l^2 d)]$, and so
\[
 \frac{\mu_0(\l) \Phi_0'(\l) \nu_1(\l) }{\mu_0(\l) \Phi_0(\l) \nu_1(\l) } \geq \frac{ (c_\l^4/d^2) \| \Phi_0'(\l) \| }{ \| \Phi_0(\l) \| } \geq \frac{ c_{-1}^4  \| \Phi_0'(\l) \|}{d^2}, \quad \forall \l \in [-1,0]. 
\]
Therefore, if $E_\eta[ \| \Phi'(0) \| ] = \infty$ then it follows from the monotone convergence theorem that both sides of \eqref{t0form} are infinite. 

Since we are assuming that the random walk is recurrent or transient to the right then the matrices $\Phi_k(0)$ are all stochastic, and thus $\nu_k(0) = \frac{1}{d} \one$ for all $k \in \Z$. Therefore, the formula for $t_0$ in \eqref{t0form} simplifies to 
\begin{equation}\label{eq:eqt0form2}
 t_0 = E_\eta\left[ \frac{\mu_0(0) \Phi_0'(0) \one }{\mu_0(0) \Phi_0(0) \one } \right]
= E_\eta\left[ \frac{\mu_0(0) \Phi_0'(0) \one }{\mu_0(0) \one } \right]
= E_\eta\left[ \mu_0(0) \Phi_0'(0) \one  \right] = E_\eta\left[ E_\w^{\mu_0(0)}[ T_1 ] \right].
\end{equation}
Finally, the proof of the law of large numbers for $T_n/n$ in \cite{rSCLT} gives a formula for the limit, and translating this formula into our notation we obtain
\[
 \lim_{n\ra\infty} \frac{T_n}{n} = E_\eta\left[ E_\w^{\mu_0(0)}[ T_1 ] \right], \quad \P^{\pi}\text{ - a.s.}
\]
\end{proof}

The following lemma characterizes the zero set of the rate function $J_\eta(t)$ and is consistent with the corresponding result for nearest-neighbor RWRE in \cite{cgzLDP}. 

\begin{lem}\label{qrfpropTn}
The quenched rate function for the hitting times $J_\eta$ has the following properties.
\begin{enumerate}
 \item If $\lim_{n\ra\infty} X_n = -\infty$, then $\inf_t J_\eta(t) > 0$. 
 \item If the RWRE is recurrent or transient to the right, then 
\begin{enumerate}
 \item If $\vp = 0$, then $J_\eta(t) > 0$ for all $t<\infty$ but $\inf_t J_\eta(t) = \lim_{t \ra\infty} J_\eta(t) = 0$. 
 \item If $\vp > 0$ and $\lcrit(\eta) = 0$ then $J_\eta(t) = 0 \iff t \geq t_0 = 1/\vp$. 
 \item If $\vp > 0$ and $\lcrit(\eta) > 0$, then $J_\eta(t) = 0 \iff t = t_0 = 1/\vp$. 
\end{enumerate}
\end{enumerate}
\end{lem}

\begin{proof}
 To prove the first part of the Lemma, note that $\inf_t J_\eta(t) = - \L_\eta(0)$ and so we need to show that $\L_\eta(0) < 0$ when the RWRE is transient to the left. To this end, note that if the RWRE is transient to the left then $P_\w^\pi(T_1 < \infty) < 1$ for $\eta$-a.e.\ environment $\w$ and any distribution $\pi$ on the starting height (here we are using Assumption \ref{easm}). Therefore, 
\[
 \L_\eta(0) = E_\eta[ \log( \mu_0(0) \Phi_0(0) \one ) ] = E_\eta[ \log P_\w^{\mu_0(0)}( T_1 < \infty) ] < 0. 
\]

The second part of the Lemma follows easily from the fact that $J_\eta(t)$ is the Legendre transform of the differentiable function $\L_\eta(\l)$, the fact that $t_0 = 1/\vp$,
and the definition of $t_0$ in \eqref{t0def}. 
\end{proof}

\subsection{Upper bound}

Since we are only proving a weak large deviation principle, the properties of $J_\eta$ in Lemma \ref{Jetaprop} imply that to prove the quenched large deviation upper bound it will be enough to show that
\begin{equation}\label{qubr}
 \limsup_{n\ra\infty} \frac{1}{n} \log P_\w^\pi( T_n \in [nt,\infty) ) \leq - \sup_{\l \geq 0} ( \l t - \L_\eta(\l) ),  \quad \forall t \in [t_0,\infty), \quad \eta\text{ - a.s.},
\end{equation}
and
\begin{equation}\label{qubl}
 \limsup_{n\ra\infty} \frac{1}{n} \log P_\w^\pi( T_n \leq nt ) \leq - \sup_{\l \leq 0} (\l t - \L_\eta(\l)) , \quad \forall t \in [1,t_0],  \quad \eta\text{ - a.s.}
\end{equation}
To show \eqref{qubr}, Chebychev's inequality implies that for any $\l\geq 0$,
\[
 P_\w^\pi( T_n  \geq [nt,\infty) ) \leq e^{-\l n t} E_\w^\pi\left[ e^{\l T_n} \ind{T_n <\infty} \right].
\]
Then, applying Lemma \ref{qlmgflim} and then optimizing over $\l\geq 0$ proves \eqref{qubr}. The proof of \eqref{qubl} is similar and therefore ommitted.

\subsection{Lower bound}


For the proof of the quenched large deviations lower bound we will need the following Lemma. 

\begin{lem}\label{ltMlem}
If $t > 1$,
then for all $M > t+2$ there exists a $\l_{t,M}$ such that $\L_{\eta,M}'(\l_{t,M}) = t$.
\end{lem}
\begin{proof}
Since $\L_{\eta,M}(\l)$ is convex and continuously differentiable, it is enough to show that 
\begin{equation}\label{compact}
 \lim_{\l \ra -\infty} \l t - \L_{\eta,M}(\l) = -\infty \quad \text{and} \quad \lim_{\l \ra \infty} \l t - \L_{\eta,M}(\l) = - \infty, \quad \forall t \in (1,M-2). 
\end{equation}
As in \eqref{Letalbub}, Assumption \ref{easm} implies that $\L_{\eta,M}(\l) \geq \l + \log \k$ for all $\l \leq 0$.
This is enough to prove the first limit in \eqref{compact} for $t>1$. To prove the second limit in \eqref{compact}, for any $\l \geq 0$ and any distribution $\pi$ on $[d]$ note that
\[
 E_\w^\pi[ e^{\l T_1} \ind{T_1 \leq M} ] \geq e^{\l (M-2)} P_\w^\pi( T_1 \in[M-2,M] ) \geq e^{\l (M-2)} \k^M, 
\]
where the last inequality follows from Assumption \ref{easm}. 
This implies that $\L_{\eta,M}(\l) \geq \l(M-2) + M \log \k$ for all $\l \geq 0$ which is enough to prove the second limit in \eqref{compact} when $t < M-2$.  
\end{proof}

For the large deviations lower bound, it will be enough to show that
\begin{equation}\label{qlbdlim}
 \lim_{\d\ra 0} \liminf_{n\ra\infty} P_\w^\pi ( | T_n - nt | < n\d ) \geq - J_\eta(t), \quad \forall t > 1, \quad \eta \text{ - a.s.}
\end{equation}
We will follow a change of measure argument that is a minor modification of the one in \cite[pp. 76-78]{cgzLDP}. 
Fix $M > \max\{ N_\k, t+2 \}$ and $\l \in \R$, and define the probability measure $Q_{\w,n}^{\l,M}$ on paths up to time $T_n$ with $\tau_k \leq M$ for all $k\leq n$ by
\begin{equation}\label{QwnlMdef}
 \frac{d Q_{\w,n}^{\l,M} }{d P_\w^\pi} = \frac{1}{Z_{n,\w,\l,M}} e^{\l T_n } \ind{\tau_k \leq M, \, k=1,2,\ldots n}, 
\quad \text{where } 
Z_{n,\w,\l,M} = E_\w^\pi\left[  e^{\l T_n } \ind{\tau_k \leq M, \, k=1,2,\ldots n} \right]. 
\end{equation}
Then,
\begin{align}
 P_\w^\pi( |T_n - nt| < n\d ) &\geq P_\w^\pi ( |T_n - nt | < n \d , \, \tau_k \leq M, \, k=1,2,\ldots, n ) \nonumber \\
&= Z_{n,\w,\l,M} E_{Q_{\w,n}^{\l,M}}\left[ e^{ - \l T_n} \ind{ |T_n - nt| < n\d } \right] \nonumber \\
&\geq  Z_{n,\w,\l,M} e^{- \l (nt \pm \d n)} Q_{\w,n}^{\l,M}( |T_n - nt| < n\d ),
\end{align}
where the $\pm$ sign in the last line depends on whether or not $\l \geq 0$.
Then, since Lemma \ref{tqlmgflim} implies that $\lim_{n\ra\infty} n^{-1} \log Z_{n,\w,\l,M} = \L_{\eta,M}(\l)$ we conclude that
\begin{equation} \label{eq:tcom}
 \liminf_{n\ra\infty} \frac{1}{n} \log P_\w^\pi( |T_n - nt| < n\d )
\geq - \l (t \pm \d) + \L_{\eta,M}(\l) + \liminf_{n\ra\infty} \frac{1}{n} \log Q_{\w,n}^{\l,M}( |T_n - nt| < n\d ). 
\end{equation}
Now, let $\l_{t,M}$ be chosen as in Lemma \ref{ltMlem} so that $\L_{\eta,M}'(\l_{t,M}) = t$. 
We claim that this choice of $\l_{t,M}$ implies that
\begin{equation}\label{eq:QwnlMLLN}
  \lim_{n\ra\infty} Q_{\w,n}^{\l_{t,M},M}( |T_n - nt| < n\d ) = 1, \quad \forall \d>0.
\end{equation}
To see this, note that for any $h>0$ Chebychev's inequality and the definition of $Q_{\w,n}^{\l,M}$ imply that
\begin{align*}
 Q_{\w,n}^{\l_{t,M},M}( T_n >  n(t+\d) )
&= e^{- h n(t+\d)} \frac{1}{Z_{n,\w,\l_{t,M},M}} E_{\w}^\pi \left[  e^{(\l_{t,M} + h) T_n } \ind{ \tau_k \leq M, \, k=1,2,\ldots n} \right].
\end{align*}
Then, Lemma \ref{tqlmgflim} implies that
\[
 \limsup_{n\ra\infty} \frac{1}{n} \log Q_{\w,n}^{\l_{t,M}}( T_n \geq  n(t+\d) ) \leq - h (t+\d)- \L_{\eta,M}(\l_{t,M}) + \L_{\eta,M}(\l_{t,M}+h).
\]
Since $\L_{\eta,M}'(\l_{t,M}) = t$, then for $h>0$ small enough (depending on $\d$) the right side above is strictly negative and so $Q_{\w,n}^{\l_{t,M}}( T_n \geq  n(t+\d) ) $ decays exponentially fast in $n$.
A similar argument shows that that $Q_{\w,n}^{\l_{t,M}}( T_n \leq  n(t-\d) )$ also decays exponentially fast in $n$ and thus \eqref{eq:QwnlMLLN} holds.

If we define $J_{\eta,M}(t) = \sup_\l (\l t - \L_{\eta,M}(\l) )$ to be the Legendre dual of $\L_{\eta,M}$, then the choice of $\l_{t,M}$ implies that $J_{\eta,M}(t) = \l_{t,M} t - \L_{\eta,M}(\l_{t,m})$. 
Therefore, \eqref{eq:tcom} and \eqref{eq:QwnlMLLN} imply that 
\[
 \lim_{\d \ra 0} \liminf_{n\ra\infty} \frac{1}{n} \log P_\w^\pi( |T_n - nt| < n\d ) \geq - \lim_{\d \ra 0} (J_{\eta,M}(t) \pm \d \l_{t,M} ) = - J_{\eta,M}(t).
\]
$\L_{\eta,M}(\l)$ is non-decreasing in $M$, and therefore $J_{\eta,M}(t)$ is non-increasing in $M$.
Thus, in order to finish 
proof of \eqref{qlbdlim}
we need to prove that
\begin{equation}\label{JMlim}
 \lim_{M\ra\infty} J_{\eta,M}(t) = J_\eta(t).
\end{equation}
Since $J_{\eta,M}(t)$ is non-increasing in $M$, we can define $J_{\eta,\infty}(t) := \lim_{M\ra\infty} J_{\eta,M}(t) \geq J_\eta(t)$.
Note that it follows from Lemma \ref{ltMlem} that $J_{\eta,\infty}(t) < \infty$ for any $t> 1$.
Then, for any $t > 1$ and $M<\infty$ define $K_{M,t} := \{ \l: \, \l t - \L_{\eta,M}(\l) \geq J_{\eta,\infty}(t) \}$.
Since $\L_{\eta,M}(\l)$ is non-decreasing in $M$, it follows that the sets $K_{M,t}$ are nested and decreasing. Also, \eqref{compact} implies that $K_{M,t}$ is compact for all large $M$. Therefore, we can conclude that there exists a $\l_{t,\infty} \in \bigcap_M K_{M,t}$, and thus
\[
 J_{\eta,\infty}(t) \leq \lim_{M\ra\infty} \l_{t,\infty} t - \L_{\eta,M}(\l_{t,\infty}) = \l_{t,\infty} t - \L_{\eta}(\l_{t,\infty}) \leq J_\eta(t).
\]
Since we showed previously that $J_{\eta,\infty}(t) \geq J_\eta(t)$, this completes the proof of \eqref{JMlim} and thus also the proof of the large deviations lower bound.

\section{Proof of the averaged LDP for hitting times}\label{sec:aldpTn}

The main goal of this section is to prove the averaged large deviation principle for the hitting times as stated in Theorem \ref{th:aldpTn}. However, before giving the proof of Theorem \ref{th:aldpTn} we must first study some properties of the rate function $\mathbb{J}_\eta(t)$. 

\subsection{Properties of the averaged rate function for hitting times}

Recall that the averaged rate function for hitting times is defined by the variational formula in \eqref{eq:arfTn} involving the specific relative entropy function $h(\cdot|\eta)$. 
It is known that $h(\a|\eta) < \infty$ only if $\a \in M_1^s(\Omega_\k)$, but it will be useful below to show that there is an even smaller subset of $M_1^s(\Omega_\k)$ where the specific relative entropy is finite. 
To this end, let $\mathcal{M}_\eta$ denote the set of stationary measures $\a$ with $\supp \a \subset \Sigma_\eta^\Z$ (recall the definition of $\Sigma_\eta$ from Lemma \ref{PhiUboundlem}). 
\begin{lem}\label{lem:hfinite}
 If $\eta$ is locally equivalent to the product of its marginals, then $h(\a|\eta)<\infty$ implies that $\a \in \mathcal{M}_\eta = \{ \a \in M_1^s(\Omega_\k): \, \supp \a \subset \Sigma_\eta^\Z\}$. 
\end{lem}
\begin{proof}
 Recall that the specific relative entropy is defined by $h(\a|\eta) = \sup_{n} \frac{1}{n} H(\a|\eta)\bigr|_{\mathcal{G}_n} $, where $\mathcal{G}_n = \s(\w_x, \, x=1,2,\ldots n)$ 
and $H$ is the general relative entropy function defined by 
\[
 H(\s|\pi) = 
\begin{cases}
 \int f \log f \, d\pi & \text{if } f= \frac{d\s}{d\pi} \text{ exists}\\
 \infty & \text{otherwise}. 
\end{cases}
\]
If $\a \notin \mathcal{M}_\eta$ then it is clear that $\a( (\w_1,\w_2,\ldots, \w_n) \in \Sigma_\eta^n ) < 1$ for some $n<\infty$, and since $\eta$ is locally equivalent to the product of it's marginals this implies that $\a$ is not absolutely continuous with respect to $\eta$ when restricted to $\mathcal{G}_n$. Thus $h(\a|\eta) \geq H(\a|\eta) \bigr|_{\mathcal{G}_n} = \infty$. 
\end{proof}

We will also need the following lemma which extends the definition of $\L_\eta(\l)$ in Lemma \ref{qlmgflim} from ergodic to stationary measures. 
\begin{lem}\label{Laldef}
 Let $\a \in M_1^s(\Omega_\k)$, then we can define 
\[
 \L_\a(\l) = 
\begin{cases}
 E_\a\left[ \log( \mu_0(\l) \Phi_0(\l) \one) \right] & \text{if } \l \leq \lcrit(\a)\\ 
\infty & \text{otherwise},
\end{cases}
\]
where $\lcrit(\a) = \sup\{ \l : \, \a( \|\Phi_0(\l) \| < \infty ) = 1 \}$. Moreover, if $\a \in \mathcal{M}_\eta$ then $\lcrit(\a) \geq \lcrit(\eta)$, and for any initial distribution $\pi$ (even depending on the environment $\w$)
\begin{equation}\label{qlmgfL1lim}
 \lim_{n\ra\infty} E_\a\left[ \frac{1}{n} \log E_\w^\pi\left[ e^{\l T_n} \ind{T_n<\infty} \right] \right] = \L_\a(\l). 
\end{equation}
\end{lem}
\begin{proof}
 Since $\a$ is stationary, it is a convex combination of ergodic measures on $\Omega_\k$ \cite[Theorem 6.6]{vPT}.
Then since $\L_\eta(\l)$ is well defined for each $\eta \in M_1^e(\Omega_\k)$ it is clear that the definition of $\L_\a(\l)$ above makes sense since the vectors $\mu_0(\l)$ are defined $\a$-a.s.\ when $\l \leq \lcrit(\a)$. Also, since the uniform bounds on $\Phi_0(\l)$ only depend on the fact that $\w \in \Omega_\k$ and $\|\Phi_k(\l)\|<\infty$ for all $k$, then it follows that $c_\l \leq \Phi_0(\l)(i,j) \leq 1/c_\l$ for all $\l \leq \lcrit(\a)$, $\a$-a.s. 
Following the proof of Lemma \ref{qlmgflim}, we see that these uniform bounds imply that \eqref{eq:qlmgfUbound} still holds $\a$-a.s.
In particular, taking expectations of \eqref{eq:qlmgfUbound} gives 
\begin{equation}\label{eq:qlmgfUL1bound}
\left| \frac{1}{n} E_\a\left[ \log E_\w^\pi[ e^{\l T_n} \ind{T_n < \infty} ]  \right] - \L_\a(\l) \right|
\leq \frac{2}{(1-c_\l^4) c_\l^{10} n},
\end{equation}
from which \eqref{qlmgfL1lim} follows easily. 
\end{proof}

\begin{lem}\label{lem:acont}
  If $\eta$ satisfies Assumptions \ref{asmplldp} and \ref{asmlocal}, then
the map 
\[
 (\l,\a) \mapsto E_\a[ \log( \mu_0(\l) \Phi_0(\l) \one ) ] =  \L_\a(\l)
\]
is jointly continuous on $(-\infty,\lcrit(\eta)) \times \mathcal{M}_\eta$ and lower semicontinuous on $(-\infty,\lcrit(\eta)] \times \mathcal{M}_\eta$,
where $\mathcal{M}_\eta \subset M_1(\Omega_\k)$
is equipped with the induced topology of weak convergence of probability measures.
\end{lem}
\begin{proof}
Recall the definition of the truncated moment generating functions $\Phi_{k,M}(\l)$ given in \eqref{PhikMdef}. 
For any $M,n<\infty$ and $i \in [d]$ it is easy to see that the function 
\begin{equation}\label{eq:truncTn}
\begin{split}
 (\l,\a) \mapsto & E_\a \left[ \frac{1}{n} \log(e_i \Phi_{[0,n-1],M}(\l) \one) \right] \\
&= E_{\a}\left[ \frac{1}{n} \log E_\w^{(0,i)}\left[ e^{\l T_n} \ind{\tau_k \leq M, \, k=1,2,\ldots n} \right]\right] =: \L_{\a,M,n,i}(\l) 
\end{split}
\end{equation}
is jointly continuous on $\R \times \mathcal{M}_\eta$ since the inner quenched expectation can be expressed as the sum over finitely many possible paths. 
We would like to show that if $(\l,\a) \in (-\infty,\lcrit(\eta))\times \mathcal{M}_\eta$ then $\L_\a(\l)$ can be approximated by 
$\L_{\a,M,n,i}(\l)$
for sufficiently large $n$ and $M$. 
To this end, we first need to be able to give a uniform error bound on the difference between the entries of $\Phi_k(\l)$ and $\Phi_{k,M}(\l)$. 
\begin{align*}
 0 \leq \Phi_k(\l)(i,j) - \Phi_{k,M}(\l)(i,j) &= E_\w^{(k,i)}\left[ e^{\l T_{k+1}} \ind{M < T_{k+1} < \infty, \, Y_{T_{k+1}} = j} \right] \\
&\leq e^{(\l-\lcrit(\eta)) M}  E_\w^{(k,i)}\left[ e^{\lcrit(\eta) T_{k+1} } \ind{T_{k+1} < \infty, \, Y_{T_{k+1}} = j} \right].
\end{align*}
Therefore, 
Lemma \ref{PhiUboundlem} implies that
\[
 \| \Phi_k(\l) - \Phi_{k,M}(\l) \| \leq e^{(\l - \lcrit(\eta))M} \frac{2d}{\k}, \quad \forall \l \leq \lcrit(\eta), \, \w \in \Sigma_\eta^\Z. 
\]
Using this bound and the fact that $\|\Phi_{k,M}(\l) \| \leq \| \Phi_k(\l) \| \leq 2d/\k$ we can then obtain that 
\begin{equation}\label{eq:truncprod}
 \| \Phi_{[0,n-1]}(\l) - \Phi_{[0,n-1],M}(\l) \| 
\leq n \left( \frac{2d}{ \k } \right)^n e^{(\l-\lcrit(\eta))M}
, \quad \forall \l \leq \lcrit(\eta), \, \w \in \Sigma_\eta^\Z. 
\end{equation}

Assumption \ref{easm} implies that $e_i \Phi_{[0,n-1],M}(\l) \one \geq \k^n e^{\l n}$. 
Since $|\log(x) - \log(y)| \leq (1/\d) |x-y|$ for $x,y \geq \d$, this together with \eqref{eq:truncprod} implies that 
\begin{align*}
\left| \log( e_i \Phi_{[0,n-1]}(\l) \one) -  \log( e_i \Phi_{[0,n-1],M}(\l) \one) \right| 
&\leq \frac{1}{\k^n e^{\l n} }  \left| e_i \Phi_{[0,n-1]}(\l) \one - e_i \Phi_{[0,n-1],M}(\l) \one \right| \\
&\leq n \left( \frac{2d}{\k^2 e^\l } \right)^n e^{(\l-\lcrit(\eta))M},
\end{align*}
for all $\l \leq \lcrit(\eta)$ and all $\w \in \Sigma_\eta^\Z$. 
Combining this with \eqref{eq:qlmgfUL1bound} we can conclude that for any $\l < \lcrit(\eta)$ and $\a \in \mathcal{M}_\eta$,
\begin{align*}
 | \L_\a(\l) - \L_{\a,M,n,i}(\l) | 
&\leq \frac{2}{n (1-c_\l^4)c_\l^{10} } + \left( \frac{2d}{\k^2 e^\l } \right)^n e^{(\l-\lcrit(\eta))M}.
\end{align*}
Thus, by first taking $n$ sufficiently large and then taking $M$ large enough (depending on $n$) we can approximate $\L_\a(\l)$ uniformly well by $\L_{\a,M,n,i}(\l)$
on the set $[\l',\l'']\times \mathcal{M}_\eta$ for any $\l' \leq \l'' < \lcrit(\eta)$. 
Since $(\l,\a) \mapsto \L_{\a,M,n,i}(\l)$ is jointly continuous this then implies that $(\l,\a) \mapsto \L_\a(\l)$ is also jointly continuous as claimed.  

Finally, to prove lower semicontinuity at $(\lcrit(\eta),\a)$, let $(\l_n, \a_n) \ra (\lcrit(\eta),\a)$. Since $\l\mapsto \L_{\a'}(\l)$ is non-decreasing and continuous for any  $\a' \in \mathcal{M}_\eta$, it follows that 
\[
 \liminf_{n\ra\infty} \L_{\a_n}(\l_n) \geq \lim_{\d\ra 0} \lim_{n\ra\infty} \L_{\a_n}(\lcrit(\eta) -\d) = \lim_{\d\ra 0} \L_\a(\lcrit(\eta) - \d) = \L_\a(\lcrit(\eta)). 
\]
Note that in the second to last equality we used the continuity away from $\lcrit(\eta)$ that we proved above.  
\end{proof}

Recall that the averaged rate function is defined by the variational representation in \eqref{eq:arfTn}. The key to proving the averaged large deviation principle with this variational formula for the rate function is the following lemma which gives an alternative formula for $\mathbb{J}_\eta(t)$ as a Legendre transform. 

\begin{lem}\label{lem:arflt}
 Let the distribution on environments $\eta$ satisfy Assumptions \ref{easm} - \ref{asmlocal}. 
Then, 
\begin{equation}\label{eq:arflt}
 \mathbb{J}_\eta(t) = \sup_\l \{ \l t - \mathbf{\L}_\eta(\l) \},
\quad\text{where}\quad
\mathbf{\L}_\eta(\l) := \sup_{\a \in M_1^s(\Omega_\k)} \left\{ \L_\a(\l) - h(\a|\eta) \right\}. 
\end{equation}
Moreover, $\mathbf{\L}_\eta(\l)$ is a convex, non-decreasing, lower semicontinuous function and $\mathbf{\L}_\eta(\l) < \infty$ if and only if $\l \leq \lcrit(\eta)$. 
\end{lem}

Before giving the proof of Lemma \ref{lem:arflt}, note that together with standard properties of Legendre transforms it implies the following Corollary. 

\begin{cor}\label{cor:arfinf}
 Let the distribution on environments $\eta$ satisfy Assumptions \ref{easm} - \ref{asmlocal}. 
Then, $\mathbb{J}_\eta(t)$ is a convex function in $t$ and 
\begin{equation}\label{eq:arfinf1}
 \inf_{s\leq t} \mathbb{J}_\eta(s) = \sup_{\l < 0} \left\{ \l t - \mathbf{\L}_\eta(\l) \right\}
\end{equation}
and if $\lcrit(\eta)>0$ then 
\begin{equation}\label{eq:arfinf2}
 \inf_{s\geq t} \mathbb{J}_\eta(s) = \sup_{0 \leq \l < \lcrit(\eta)} \left\{ \l t - \mathbf{\L}_\eta(\l) \right\}. 
\end{equation}
\end{cor}

\begin{proof}
The fact that $\mathbb{J}_\eta(t)$ is a convex function follows from the representation in \eqref{eq:arflt} of $\mathbb{J}_\eta(t)$ as the Legendre transform of $\mathbf{\L}_\eta(\l)$. 
The equalities \eqref{eq:arfinf1} and \eqref{eq:arfinf2} follow from standard properties of Legendre transforms and are thus ommitted (for more details see the proof of the corresponding equalities in the proof of Proposition 3 in \cite{dgzLDPH}). 
\end{proof}

\begin{proof}[Proof of Lemma \ref{lem:arflt}]
Since $\L_\a(\l)$ is convex and non-decreasing in $\l$ for any $\a \in M_1^s(\Omega_\k)$, it follows that $\mathbf{\L}_\eta(\l)$ is also convex and non-decreasing in $\l$. 
The definition of $\mathbf{\L}_\eta(\l)$ implies that $\mathbf{\L}_\eta(\l) \geq \L_\eta(\l)$, and thus it follows that $\mathbf{\L}_\eta(\l) = \infty$ for any $\l > \lcrit(\eta)$. 
On the other hand, if $\l \leq \lcrit(\eta)$ then 
\[
 \mathbf{\L}_\eta(\l) = \sup_{\a \in \mathcal{M}_\eta} \left\{ \L_\a(\l) - h(\a|\eta) \right\} \leq \sup_{\a \in \mathcal{M}_\eta} \L_\a(\l) \leq \log(d'/c_\l), 
\]
where the first equality follows from the fact that $h(\a|\eta) = \infty$ for $\a \notin \mathcal{M}_\eta$ and the last equality follows from the uniform upper bound in \eqref{eq:PhiUboundU} on the entries of $\Phi_0(\l)$ for environments $\w \in \Sigma_\eta^\Z$. 
This shows that the domain of $\mathbf{\L}_\eta(\l)$ is $(-\infty, \lcrit(\eta)]$. 
For any $\a \in \mathcal{M}_\eta$ the function $\l \mapsto \L_\a(\l)$ is continuous on $(-\infty,\lcrit(\eta)]$
(continuity follows from Lemma \ref{lem:acont} when on $(-\infty,\lcrit(\eta))$ and by monotone convergence at $\l = \lcrit(\eta)$). 
Since $ \mathbf{\L}_\eta(\l)$ is the supremum of a family of continuous functions this implies that $\mathbf{\L}_\eta(\l)$ is lower semicontinuous.

We have thus shown the claimed properties of $\mathbf{\L}_\eta(\l)$ and it remains to show that $\mathbb{J}_\eta(t)$ is the Legendre transform of $\mathbf{\L}_\eta(\l)$. 
First, note that the supremum in the definition of $\mathbf{\L}_\eta(\l)$ can be restricted to $\a \in \mathcal{M}_\eta$ 
by Lemma \ref{lem:hfinite}. 
Next, we claim that
\begin{align}
 \sup_{\l} \left\{ \l t - \mathbf{\L}_\eta(\l) \right\}
&= \sup_{\l \leq \lcrit(\eta)} \inf_{\a \in \mathcal{M}_\eta} \left\{ \l t - \L_\a(\l) + h(\a|\eta) \right\} \nonumber \\
&= \inf_{\a \in \mathcal{M}_\eta} \sup_{\l \leq \lcrit(\eta)} \left\{ \l t - \L_\a(\l) + h(\a|\eta) \right\}. \label{eq:infsup}
\end{align}
The restriction of the supremum to $\l \leq \lcrit(\eta)$ in the first equality is justified by the fact that $\mathbf{\L}_\eta(\l) = \infty$ for $\l > \lcrit(\eta)$ and the interchange of the supremum and infimum in the second equality above follows from a minimax theorem \cite[Theorem 4.2']{sMinimax} since the function $(\l,\a)\mapsto  \l t - \L_\a(\l) + h(\a|\eta)$ is concave in $\l$, convex in $\a$, and the set $\mathcal{M}_\eta$ is compact. 

To finish the proof of the lemma, we need to show that the infimum in \eqref{eq:infsup} can be restricted to $\a \in M_1^e(\Omega_\k)$. To this end, first note that since the function $\a \mapsto h(\a|\eta) + \sup_{\l < \lcrit(\eta)} \{ \l t - \L_\a(\l) \}$ is lower semicontinuous and the set $\mathcal{M}_\eta$ is compact, there exists an $\a' \in \mathcal{M}_\eta$ that achieves the infimum in \eqref{eq:infsup}. 
As in the proof of Lemma \ref{Jetaprop}, the uniform ellipticity assumptions imply that $\L_\eta(\l) \geq \l + \log \k$ for all $\l\leq 0$. Thus, for any $t>1$ the supremum in \eqref{eq:infsup} can be restricted to $\l \geq K_t := \log \k/(t-1)$. 
Since $\l \mapsto \l t - \L_{\a'}(\l)$ is concave there exists a pair $(\a',\l') \in \mathcal{M}_\eta \times [K_t,\lcrit(\eta)]$ such that 
\begin{align}\label{eq:optimalpair}
\begin{split}
\inf_{\a \in \mathcal{M}_\eta} \sup_{\l \leq \lcrit(\eta)} \left\{ \l t - \L_\a(\l) + h(\a|\eta) \right\}
&= \l' t - \L_{\a'}(\l') + h(\a'|\eta). 
\end{split}
\end{align}
By Assumption \ref{asmlocal} there exists a sequence of ergodic measures $\a_n \ra \a'$ with $h(\a_n|\eta) \ra h(\a'|\eta)$. For each $\a_n$ let $\l_n \in [K_t,\lcrit(\eta)]$ be such that $\sup_{\l\leq \lcrit(\eta)} \{ \l t - \L_{\a_n}(\l) \} = \l_n t - \L_{\a_n}(\l_n)$. 
Thus, by passing to a subsequential limit we can assume without loss of generality that $\l_n \ra \l^*$ for some $\l^* \in [K_t,\lcrit(\eta)]$.
Finally, by the lower semicontinuity proved in Lemma \ref{lem:acont} we can conclude that 
\begin{align}
 \inf_{\a \in M_1^e(\Omega_\k)} \sup_{\l \leq \lcrit(\eta)} \left\{ \l t - \L_\a(\l) + h(\a|\eta) \right\}
&\leq \liminf_{n \ra \infty} \l_n t - \L_{\a_n}(\l_n) + h(\a_n|\eta) \nonumber \\
&\leq \l^*t - \L_{\a'}(\l^*) + h(\a'|\eta). \label{eq:suboptimalpair}
\end{align}
Combining \eqref{eq:optimalpair} and \eqref{eq:suboptimalpair} we conclude that 
\[
 \inf_{\a \in M_1^e(\Omega_\k)} \sup_{\l \leq \lcrit(\eta)} \left\{ \l t - \L_\a(\l) + h(\a|\eta) \right\}
\leq \inf_{\a \in M_1^s(\Omega_\k)} \sup_{\l < \lcrit(\eta)} \left\{ \l t - \L_\a(\l) + h(\a|\eta) \right\}. 
\]
Since the reverse inequality is trivial, recalling \eqref{eq:infsup} we can conclude for $t>1$ that 
\[
 \sup_\l \{ \l t - \mathbf{\L}_\eta(\l) \} = \inf_{\a \in M_1^e(\Omega_\k)} \sup_{\l < \lcrit(\eta)} \left\{ \l t - \L_\a(\l) + h(\a|\eta) \right\}
= \inf_{\a \in M_1^e(\Omega_\k)} J_\a(t) + h(\a|\eta) = \mathbb{J}_\eta(t). 
\]

It is easy to see that both sides of \eqref{eq:arflt} are infinite when $t < 1$ since $\mathbf{\L}_\eta(\l) \leq \l$ for $\l \leq 0$ and $J_\a(t) = \infty$ for $t<1$ and any $\a$. Thus, it remains only to show that \eqref{eq:arflt} holds when $t=1$. 
We will show this by a slight variation in the minimax argument used above in the case when $t>1$. First, note that since $\mathbf{\L}_\eta(\l)$ is convex and $\l + \log \k \leq \mathbf{\L}_\eta(\l) \leq \l$ for $\l \leq 0$, it follows that 
\begin{equation}\label{lbLlim}
 \sup_\l \{ \l - \mathbf{\L}_\eta(\l) \} = \lim_{\l \ra -\infty} \{ \l - \mathbf{\L}_\eta(\l) \}. 
\end{equation}
Secondly, note that the continuity of $\a \mapsto \L_\a(\l)$ implies that the supremum in the definition of $\mathbf{\L}_\eta(\l)$ can be restricted to $\a \in M_1^e(\Omega_\k)$ if $\l < \lcrit(\eta)$ (Note that here we are also using Assumption \ref{asmlocal} here to approximate the entropy for stationary $\a$ by entropy of ergodic $\a$.). Combining these two facts we obtain that 
\begin{equation}\label{superg}
 \sup_\l \{ \l - \mathbf{\L}_\eta(\l) \} = \sup_{\l \leq -1} \inf_{\a \in M_1^e(\Omega_\k)} \left\{ \l - \L_\a(\l) + h(\a|\eta)\right \}. 
\end{equation}
For convenience of notation define $f_\eta(\l,\a) = \l - \L_\a(\l) + h(\a|\eta)$. Note that as in \eqref{lbLlim}, the upper and lower bounds on $\L_\a(\l)$ in \eqref{Letalbub} imply that 
\begin{equation}\label{lLlim}
 \sup_\l \{ \l - \L_\a(\l) \} = \lim_{\l \ra -\infty} \{ \l - \L_\a(\l) \}, \quad \forall \a \in M_1^s(\Omega_\k). 
\end{equation}
Thus, we can define $f_\eta(-\infty,\a) = \lim_{\l \ra -\infty} f_\eta(\l,\a)$ so that we may write
\[
 \sup_\l\{ \l - \mathbf{\L}_\eta(\l) \} = \sup_{\l \in [-\infty, -1]} \inf_{\a \in M_1^e(\Omega_\k)} f(\l,\a). 
\]
Since $[-\infty,-1]$ is compact, if we can show that $f_\eta(\l,\a)$ is lower semicontinuous in $\a$ for any $\l \in [-\infty,-1]$ we will be able to apply the minimax theorem \cite[Theorem 4.2']{sMinimax} to conclude that 
\begin{equation}\label{eq:mm2}
\begin{split}
 \sup_\l\{ \l - \mathbf{\L}_\eta(\l) \}
 &= \inf_{\a \in M_1^e(\Omega_\k)} \sup_{ \l \in [-\infty,-1] } \left\{ \l - \L_\a(\l) + h(\a|\eta) \right\} \\
 &= \inf_{\a \in M_1^e(\Omega_\k)} \left\{ J_\a(1) + h(\a|\eta) \right\} = \mathbb{J}_\eta(1). 
 \end{split}
\end{equation}
(Note that we are applying a different minimax theorem here than we did above.)
Lemma \ref{lem:acont} implies that $f_\eta(\l,\a)$ is lower semicontinuous in $\a$ for any $\l \in (-\infty,\lcrit(\eta))$. To prove lower semicontinuity in $\a$ when $\l=-\infty$, if $\a_n \ra \a$ then for any fixed $\l_0 > -\infty$ 
\begin{align*}
 \liminf_{n\ra\infty} f_\eta(-\infty,\a_n) 
& = \liminf_{n\ra\infty} \lim_{\l \ra -\infty} \{ \l - \L_{\a_n}(\l) + h(\a_n|\eta) \\
&\geq \liminf_{n\ra\infty} \{ \l_0 - \L_{\a_n}(\l_0) + h(\a_n|\eta) \} \\
&\geq \l_0 - \L_\a(\l_0) + h(\a|\eta), 
\end{align*}
where in the second to last inequality we used \eqref{lLlim} and in the last inequality we used Lemma \ref{lem:acont} and the fact that $h(\a|\eta)$ is lower semicontinuous in $\a$. 
Since this is true for any $\l_0 \in (-\infty,\lcrit(\eta))$ we can conclude that 
$\liminf_{n\ra\infty} f_\eta(-\infty,\a_n) \geq \lim_{\l_0 \ra -\infty} f_\eta(\l_0,\a) = f_\eta(-\infty,\a)$. 
This justifies our application of the minimax argument in \eqref{eq:mm2} and thus finishes the proof of the Lemma. 
\end{proof}

The final property of the averaged rate function $\mathbb{J}_\eta(t)$ that we will consider is a characterization of the zero set. 
Lemma \ref{qrfpropTn} gives a description of where the quenched rate function for hitting times $J_\eta(t)$ is zero. A consequence of Lemma \ref{lem:acont} is that the averaged rate function has the same zero set.  
\begin{lem}\label{arfpropTn}
 If the measure $\eta$ on environments satisfies Assumptions \ref{easm}, \ref{asmplldp} and \ref{asmlocal}, then $\mathbb{J}_\eta(t) = 0 \iff J_\eta(t) = 0$. 
In particular, if the RWRE is recurrent or transient to the right then 
$\mathbb{J}_\eta(t)$ is non-increasing on $[1,1/\vp]$ and non-decreasing on $[1/\vp,\infty)$. 
\end{lem}
\begin{proof}
 Obviously from the definition of the averaged rate function in \eqref{eq:arfTn}, it follows that $\mathbb{J}_\eta(t) \leq J_\eta(t)$ for all $t$ and so $J_\eta(t) = 0$ implies that $\mathbb{J}_\eta(t) = 0$ also. On the other hand, assume for contradiction that $J_\eta(t) > 0$ but $\mathbb{J}_\eta(t) = 0$. Then, there exists a sequence $\a_n \in M_1^e(\Omega_\k)$ of ergodic measures such that $J_{\a_n}(t) \ra 0$ and $ h(\a_n|\eta) \ra 0$ as $n\ra\infty$. If $h(\a_n|\eta) \ra 0$, then it must be true that $\a_n \ra \eta$. However, if $J_\eta(t) > 0$ then there exists a $\l' < \lcrit(\eta)$ such that $\l' t - \L_\eta(\l') > 0$, and thus Lemma \ref{lem:acont} implies that
\[
 \liminf_{n\ra\infty} J_{\a_n}(t) \geq \liminf_{n\ra\infty} \l' t - \L_{\a_n}(\l') = \l' t - \L_\eta(\l') > 0.
\]
Since this contradicts the claim that $J_{\a_n}(t) \ra 0$ as $n\ra\infty$ this completes the proof that the zero sets of $J_\eta$ and $\mathbb{J}_\eta$ are identical. 
The final claim follows from the fact that $\mathbb{J}_\eta(t)$ is a non-negative convex function and $\mathbb{J}_\eta(1/\vp) = \mathbb{J}_\eta(t_0) = 0$. 
\end{proof}

\subsection{Upper bound}

As in the quenched case, the key to proving the large deviation uppper bound is computing the asymptotics of the averaged log moment generating functions of $T_n$. 
\begin{lem}\label{lem:almgf}
 Let the distribution on environments $\eta$ satisfy Assumptions \ref{easm}, \ref{asmplldp} and \ref{asmlocal}. Then,
\[
 \limsup_{n\ra\infty} \frac{1}{n} \log \E^\pi\left[ e^{\l T_n} \ind{T_n < \infty} \right] \leq 
\mathbf{\L}_\eta(\l), 
\quad \forall \l < \lcrit(\eta). 
\]
\end{lem}

\begin{proof}
 We begin by noting that
\begin{align*}
 \frac{1}{n} \log \E^\pi\left[ e^{\l T_n} \ind{T_n < \infty} \right]
&= \frac{1}{n} \log E_\eta\left[ E_\w^\pi\left[ e^{\l T_n} \ind{T_n < \infty} \right] \right]\\
&= \frac{1}{n} \log E_\eta \left[ \exp\left\{ \log E_\w^\pi\left[ e^{\l T_n} \ind{T_n < \infty} \right]  \right\} \right].
\end{align*}
Then, since we can approximate $\log E_\w^\pi\left[ e^{\l T_n} \ind{T_n < \infty} \right]$ by $\sum_{k=0}^{n-1} \log( \mu_k(\l) \Phi_k(\l) \one)$ with uniform error bounds given in \eqref{eq:qlmgfUbound} it follows that
\begin{equation}\label{eq:almgfUbound}
 \left| \frac{1}{n} \log \E^\pi\left[ e^{\l T_n} \ind{T_n < \infty} \right]  - \frac{1}{n} \log E_\eta\left[ \exp\left\{ \sum_{k=0}^{n-1} \log( \mu_k(\l) \Phi_k(\l) \one) \right\} \right] \right| \leq \frac{2}{(1-c_\l^4)c_\l^{10} n}.
\end{equation}
Moreover, recalling the definition of the empirical process $L_n$ in \eqref{eq:Lndef}, we can re-write the sum inside the second expectation on the left as
\begin{equation}\label{eq:Lnform}
 \sum_{k=0}^{n-1} \log( \mu_k(\l) \Phi_k(\l) \one)
= n \int_{\Omega} \log( \mu_0(\l) \Phi_0(\l) \one ) \, L_n(d\w).
\end{equation}
Recall that $L_n$ satisfies a large deviation principle on $M_1(\Omega_\k)$ with rate function $h(\cdot|\eta)$. Then Lemma \ref{lem:acont} allows us to apply a version of Varadhan's Lemma  (Lemma 4.3.6 in \cite{dzLDTA}) to conclude that for any $\l < \lcrit(\eta)$
\begin{equation}\label{eq:Vlemapp}
 \limsup_{n\ra\infty} \frac{1}{n} \log E_\eta\left[ \exp \left\{ n \int_{\Omega} \log( \mu_0(\l) \Phi_0(\l) \one ) \, L_n(d\w) \right\} \right] 
\leq \sup_{\a \in M_1^s(\Omega_\k)} \left\{ \L_\a(\l) - h(\a|\eta) \right\}. 
\end{equation}
Combining \eqref{eq:almgfUbound}, \eqref{eq:Lnform} and \eqref{eq:Vlemapp}  and recalling the definition of $\mathbf{\L}_\eta(\l)$ in \eqref{eq:arflt} finishes the proof of the Lemma. 
\end{proof}

We are now ready to prove the large deviation upper bound for Theorem \ref{th:aldpTn}. Since the rate function $\mathbb{J}_\eta(t)$ is convex and we are only proving a weak large deviation principle, it is enough to prove the large deviation upper bound for the left and right tails of the hitting times. 
To this end, note that 
$\P^\pi( T_n \leq nt ) \leq e^{-\l nt} \E^\pi [ e^{\l T_n} \ind{T_n < \infty} ]$ for any $\l < 0$. 
Thus, Lemma \ref{lem:almgf} and Corollary \ref{cor:arfinf} imply that 
\[
 \lim_{n\ra\infty} \frac{1}{n} \log \P^\pi( T_n \leq nt ) \leq - \sup_{\l < 0} \left\{ \l t - \mathbf{\L}_\eta(\l) \right\} = - \inf_{s\leq t} \mathbb{J}_\eta(s). 
\]
The large deviation upper bound for the right tails is proved similarly. In particular, since $\P^\pi( T_n \in [nt, \infty) ) \leq e^{-\l nt} \E^\pi [ e^{\l T_n} \ind{T_n < \infty} ]$ for any $\l \geq 0$, then Lemma \ref{lem:almgf} and Corollary \ref{cor:arfinf} imply that
\[
 \limsup_{n\ra\infty} \frac{1}{n} \log \P^\pi( T_n \in [nt, \infty) ) \leq - \sup_{\l \in [0,\lcrit(\eta))} \left\{ \l t - \mathbf{\L}_\eta(\l) \right\} = - \inf_{s\geq t} \mathbb{J}_\eta(s). 
\]

\subsection{Lower bound}

To prove the averaged large deviation lower bound for hitting times it will be enough to show that 
\[
 \lim_{\d\ra 0} \liminf_{n\ra\infty} \frac{1}{n} \log \P^\pi( |T_n - nt| < n\d ) = - \mathbb{J}_\eta(t), \quad \forall t > 1. 
\]
To this end, recall the definition of $ Q_{\w,n}^{\l,M} $ from \eqref{QwnlMdef} in the proof of the quenched large deviation lower bound for hitting times. $Q_{\w,n}^{\l,M}$ is a distribution on paths of the random walk up to time $T_n$ that depends on the environment $\w$ and so for any ergodic measure $\a$ on environments we may define the corresponding averaged measure $\mathbb{Q}_{\a,n}^{\l,M}( \cdot) = E_\a\left[ Q_{\w,n}^{\l,M}( \cdot) \right]$. 
Now, let $\mathcal{F}_{n,M} := \s( \{\tau_k\}_{k=1}^n, \, \{ \w_x \}_{x=-M+1}^{n-1} )$ and $ \mathcal{F}_{n,M}^\w := \s( \{ \w_x \}_{x=-M+1}^{n-1} )$. Then it is easy to see that 
\begin{equation}\label{eq:entropysum}
 H( \mathbb{Q}_{\a,n}^{\l,M} | \P^\pi ) \bigr|_{\mathcal{F}_{n,M}} = 
H(\a|\eta) \bigr|_{\mathcal{F}_{n,M}^\w} + \int_\Omega H( Q_{\w,n}^{\l,M} | P_\w^\pi ) \bigr|_{\mathcal{F}_{n,M}} \, \a(d\w), 
\end{equation}
where in the above $H(\cdot|\cdot)$ is the relative entropy function. 
The definition of the measure $Q_{\w,n}^{\l,M}$ implies that 
\[
 H( Q_{\w,n}^{\l,M} | P_\w^\pi )\bigr|_{\mathcal{F}_{n,M}}  
=  E_{Q_{\w,n}^{\l,M}}\left[ \l T_n \right] - \log E_\w^\pi[ e^{\l T_n} \ind{\tau_k \leq M, \, k=1,2,\ldots n} ]. 
\]
Then as in the proof of the quenched large deviation lower bound, choosing $\l_{t,M}$ as in Lemma \ref{ltMlem} and then applying \eqref{eq:QwnlMLLN} and Lemma \ref{tqlmgflim} we obtain that 
\[
 \lim_{n\ra\infty} \frac{1}{n} \int_\Omega H( Q_{\w,n}^{\l_{t,M},M} | P_\w^\pi ) \bigr|_{\mathcal{F}_{n,M}} \, \a(d\w) = \l_{t,M} t - \L_{\a,M}(\l_{t,M}) = J_{\a,M}(t). 
\]
Note that in taking this limit we used the fact that the measure $Q_{\w,n}^{\l,M}$ is constructed so that $Q_{\w,n}^{\l,M}( T_n/n \leq M) = 1$. 
Then, since $\limsup_{n\ra\infty} n^{-1} H(\a|\eta) \bigr|_{\mathcal{F}_{n,M}^\w} \leq h(\a|\eta)$ we conclude from \eqref{eq:entropysum} that 
\[
 \lim_{n\ra\infty} \frac{1}{n}  H( \mathbb{Q}_{\a,n}^{\l_{t,M},M} | \P^\pi ) \bigr|_{\mathcal{F}_{n,M}} = J_{\a,M}(t) + h(\a|\eta). 
\]
Finally, since \eqref{eq:QwnlMLLN} implies that $\lim_{n\ra\infty} Q_{\w,n}^{\l_{t,M},M}( |T_n - t n | < \d n ) = 1$ for any $\d>0$, it follows from \cite[Lemma 7]{cgzLDP} that 
\[
 \liminf_{n\ra\infty} \frac{1}{n} \log \P^\pi( |T_n - t n| < \d n) \geq - \left\{ J_{\a,M}(t) + h(\a|\eta) \right\}. 
\]
This inequality holds for any $\d>0$, $M<\infty$ and $\a \in M_1^e(\Omega_\k)$. Thus we conclude that
\begin{align*}
 \lim_{\d \ra 0} \liminf_{n\ra\infty} \frac{1}{n} \log \P^\pi( |T_n - t n| < \d n) 
&\geq - \inf_{\a \in M_1^e(\Omega_\k)} \lim_{M\ra\infty} \left\{ J_{\a,M}(t) + h(\a|\eta) \right\} \\
& = - \inf_{\a \in M_1^e(\Omega_\k)} \left\{ J_\a(t) + h(\a,\eta) \right\} \\
& = - \mathbb{J}_\eta(t), 
\end{align*}
where the second equality follows from \eqref{JMlim}. This completes the averaged large deviations lower bound for hitting times.

\section{Transferring a LDP from \texorpdfstring{$T_n/n$ to $X_n/n$}{time to space}}\label{sec:ldpXn}

Having proved the quenched and averaged large deviation principles for the hitting times, we now use these to prove quenched and averaged large deviation principles for the speed $X_n/n$ of the random walk. 
We begin by defining what will be the quenched and averaged rate functions for the speed. 
\begin{equation}\label{eq:srfdef}
 I_\eta(x) = 
\begin{cases}
 x J_\eta(1/x) & x > 0 \\
 \lcrit(\eta) & x = 0 \\
|x| J_\eta(1/|x|) & x<0 
\end{cases}
\quad\text{and}\quad
 \mathbb{I}_\eta(x) = 
\begin{cases}
 x \mathbb{J}_\eta(1/x) & x > 0 \\
 \lcrit(\eta) & x = 0 \\
|x| \mathbb{J}_\eta(1/|x|) & x<0. 
\end{cases}
\end{equation}

\begin{lem}\label{lem:qsrf}
 If the distribution $\eta$ on environments satisfies Assumptions \ref{asmerg} and \ref{easm}, then the function $I_\eta$ as defined in \eqref{eq:srfdef} is continuous and convex on $[-1,1]$. Moreover, if $\lcrit > 0$ then $I_\eta(x) = 0$ if and only if $x = \vp$ while if $\lcrit = 0$ then $I_\eta(x) = 0$ if and only if $x$ is in the closed interval between $0$ and $\vp$. 
\end{lem}

\begin{lem}\label{lem:asrf}
  If the distribution $\eta$ on environments satisfies Assumptions \ref{easm}-\ref{asmlocal} then the function $\mathbb{I}_\eta$ as defined in \eqref{eq:srfdef} is continuous and convex on $[-1,0]$ and $[0,1]$ separately. 
Moreover, the averaged rate function has the same zero set as the quenched rate function: $\mathbb{I}_\eta(x) = 0 \iff I_\eta(x) = 0$. 
\end{lem}

\begin{rem}
Since $\mathbb{I}_\eta(x)$ is non-negative, it follows that $\mathbb{I}_\eta(x)$ is convex on all of $[-1,1]$ if $\lcrit(\eta) = \mathbb{I}_\eta(0) = 0$. 
We suspect that $\mathbb{I}_\eta(x)$ is a convex function even when $\lcrit(\eta) > 0$, but are currently unable to prove this with our techniques. Nonetheless, convexity of the rate function is not needed to prove the averaged large deviation principle for $X_n/n$. It is likely that the techniques of Varadhan \cite{vLDP} which were later generalized by Rassoul-Agha \cite{rLDPmixing} can be used to give another proof of the averaged large deviation principle for RWRE on the strip and show that indeed the rate function is convex. 
\end{rem}

\begin{proof}[Proof of Lemma \ref{lem:qsrf}]
 If $f(x)$ is a convex function, then $x\mapsto x f(1/x)$ is also convex. Thus, the definition if $I_\eta$ implies that $I_\eta$ is continuous and convex on $[-1,0)$ and $(0,1]$ separately. We still need to show that $I_\eta$ is continuous and convex at the origin. The continuity at the origin will follow from the following two facts. 
\begin{equation}\label{eq:I0cont}
 \lim_{t\ra\infty} J_\eta(t)/t = \lcrit(\eta) \quad\text{and} \quad \lcrit(\eta) = \lcrit(\eta^{\text{Inv}}).
\end{equation}
The first assertion in $\eqref{eq:I0cont}$ follows easily from the fact that $J_\eta(t)$ is the Legendre transform of the convex function $\L_\eta(\l)$ and $\L_\eta(\l) < \infty$ if and only if $\l \leq \lcrit(\eta)$. 
To prove the second assertion in \eqref{eq:I0cont}, define the matrices $\bar{\Phi}_n(\l)$ by
\[
 \bar\Phi_n(\l)(i,j) = E_\w^{(n,i)}\left[ e^{\l T_{n-1}} \ind{T_{n-1} < \infty, \, Y_{T_{n-1}} = j} \right].
\]
Note that $\bar\Phi_n(\l)(\w)$ is $\Phi_{-n}(\l)(\w^{\text{Inv}})$
and thus Lemma \ref{lcritlem} (see also Remark \ref{lcritrem} and Lemma \ref{PhiUboundlem}) implies that
\begin{equation}\label{eq:lcritInv}
 E_\w^{(0,i)}\left[ e^{\l T_{-1}} \ind{T_{-1} < \infty} \right] < \infty  \quad \iff \quad \l \leq \lcrit(\eta^{\text{Inv}}). 
\end{equation}
Now, for any $n \geq 1$,
\begin{align*}
E_\w^{(0,i)}\left[ e^{\l T_n} \ind{ T_n < \infty} \right] 
&\geq E_\w^{(0,i)}\left[ e^{\l T_n} \ind{T_{-1} < T_n < \infty} \right] \\
&\geq \sum_j E_\w^{(0,i)}\left[ e^{\l T_{-1}} \ind{T_{-1} < T_n, \, Y_{T_{-1}} = j } \right] \Phi_{-1}(\l)(j,i) E_\w^{(0,i)}\left[ e^{\l T_n} \ind{T_n < \infty} \right].
\end{align*}
If $\l \leq \lcrit(\eta)$, then Lemma \ref{PhiUboundlem} implies that $E_\w^{(0,i)}[ e^{\l T_n} \ind{T_n < \infty} ] < \infty$ and $\Phi_{-1}(\l)(j,i) \geq c_\l$ for all $j \in [d]$. 
Thus, we obtain that 
\[
E_\w^{(0,i)}\left[ e^{\l T_{-1}} \ind{T_{-1} < T_n } \right] \leq 1/c_\l \, \quad \forall n \geq 1, \, i \in [d], \, \l \leq \lcrit(\eta).
\]
The monotone convergence theorem then implies that
$E_\w^{(0,i)}[ e^{\l T_{-1}} \ind{T_{-1}<\infty} ] \leq 1/c_\l$
for any $i\in [d]$ and $\l \leq \lcrit(\eta)$. 
Applying this to \eqref{eq:lcritInv} we obtain that $\lcrit(\eta) \leq \lcrit(\eta^{\text{Inv}})$. 
The reverse inequality follows from a symmetric argument.

To show that $I_\eta(x)$ is convex at $x=0$
it is enough to show that $I_\eta(-x) + I_\eta(x) \geq 2 I_\eta(0) = 2 \lcrit$ for any $x>0$. However, since $J_\eta(t) \geq \lcrit t - \L_\eta(\lcrit)$ for any $t$ we have that
\[
 I_\eta(-x) + I_\eta(x) = x J_{\eta^{\text{Inv}}}(1/x) + x J_\eta(1/x) 
\geq 2 \lcrit - x\left\{ \L_{\eta^{\text{Inv}}}(\lcrit ) + \L_{\eta}(\lcrit) \right\}. 
\]
Therefore, convexity at the origin will follow if we can show that $\L_{\eta^{\text{Inv}}}(\lcrit ) + \L_{\eta}(\lcrit) \leq 0$. 
To see this, first note that for any $n\geq 1$ and $\l \leq \lcrit$, 
\begin{align*}
& E_\w^{(0,i)} \left[ e^{\l T_{-n}} \ind{T_n < T_{-n} < \infty} \right] \\
&= \sum_{j,l \in [d]} E_\w^{(0,i)} \left[ e^{\l T_n} \ind{T_n < T_{-n}, \, Y_{T_n} = j} \right] E_\w^{(n,j)}\left[ e^{\l T_0} \ind{T_0 < \infty, \, Y_{T_0} = l }\right] E_\w^{(0,l)} \left[ e^{\l T_{-n}} \ind{T_{-n} < \infty} \right]. 
\end{align*}
Choosing $i = i_n(\w)$ that minimizes $E_\w^{(0,i)} \left[ e^{\l T_{-n}} \ind{T_{-n} < \infty} \right]$ we obtain that 
\begin{align}
 1 
& \geq  \sum_{j,l \in [d]} E_\w^{(0,i_n)} \left[ e^{\l T_n} \ind{T_n < T_{-n}, \, Y_{T_n} = j} \right] E_\w^{(n,j)}\left[ e^{\l T_0} \ind{T_0 < \infty, \, Y_{T_0} = l }\right] \label{prematrixub}
\end{align}
For any $k > -n$ and $\l \leq \lcrit$, define the matrices $\Phi_k^{(n)}(\l)$ by 
\[
 \Phi_k^{(n)}(\l)(i,j) = E_\w^{(k,i)}\left[ e^{\l T_{k+1}} \ind{T_{k+1} < T_{-n}, \, Y_{T_{k+1}} = j} \right]. 
\]
With this notation, 
taking logarithms in \eqref{prematrixub} and dividing by $n$ we can write
\begin{align}
 0 & \geq \frac{1}{n} \log \left( e_{i_n} \left( \Phi_0^{(n)}(\l) \cdots \Phi_{n-1}^{(n)}(\l) \right) \left( \bar{\Phi}_n(\l) \cdots \bar{\Phi}_1(\l) \right) \one \right) \nonumber \\
&= \frac{1}{n} \log \left( e_{i_n}  \Phi_0^{(n)}(\l) \cdots \Phi_{n-1}^{(n)}(\l)  \one \right) + \frac{1}{n} \log\left( \pi_n \bar{\Phi}_n(\l) \cdots \bar{\Phi}_1(\l) \one \right), \label{eq:logsum}
\end{align}
where
\[
 \pi_n =   \frac{e_{i_n} \Phi_0^{(n)}(\l) \cdots \Phi_{n-1}^{(n)}(\l)}{e_{i_n} \Phi_0^{(n)}(\l) \cdots \Phi_{n-1}^{(n)}(\l)  \one}.
\]

Now, recall the definition of the matrices $\Phi_{n,M}(\l)$  from \eqref{PhikMdef} and note that $\Phi_{k,M}(\l)(i,j) \leq \Phi_k^{(n)}(\l)(i,j) \leq \Phi_k(\l)(i,j)$ when $M \leq k+n$. Then Lemmas \ref{qlmgflim} and \ref{tqlmgflim} imply that 
\begin{equation}
 \lim_{n\ra\infty} \frac{1}{n} \log \left( e_{i_n}  \Phi_0^{(n)}(\l) \cdots \Phi_{n-1}^{(n)}(\l)  \one \right) = \L_\eta(\l), \quad \eta\text{ - a.s.} \label{condqlmgflim}
\end{equation}
Also, as in the proof of Lemma \ref{qlmgflim} we can show that 
\[
 \sup_{\pi} \left|  \frac{1}{n} \log\left( \pi \bar{\Phi}_n(\l) \cdots \bar{\Phi}_1(\l) \one \right) - \frac{1}{n} \sum_{k=1}^n \log( \bar{\mu}_k(\l) \bar{\Phi}_k(\l) \one ) \right| \leq \frac{2}{n(1-c_\l^4) c_\l^{10}} , 
\]
where $\bar{\mu}_k(\l)(\w) = \mu_{-k}(\l)(\w^{\text{Inv}})$. Thus, we can conclude that 
\begin{equation}\label{reverseqlmgflim}
 \lim_{n\ra\infty} \frac{1}{n} \log\left( \pi_n \bar{\Phi}_n(\l) \cdots \bar{\Phi}_1(\l) \one \right) = E_\eta\left[ \log( \bar{\mu}_0(\l) \bar{\Phi}_0(\l) \one ) \right] = \L_{\eta^{\text{Inv}}}(\l) , \quad \eta\text{ - a.s.}
\end{equation}
Applying \eqref{condqlmgflim} and \eqref{reverseqlmgflim} to \eqref{eq:logsum} implies that $\L_\eta(\l) + \L_{\eta^{\text{Inv}}}(\l)$ for any $\l \leq \lcrit$ which, as noted above, shows that $I_\eta$ is convex at $x=0$. 

Finally, we note that the claimed properties of the zero set of $I_\eta$ follow from the corresponding properties for the zero sets of $J_\eta$ and $J_{\eta^{\text{Inv}}}$ which can be deduced from Lemma \ref{qrfpropTn}. 
\end{proof}

\begin{proof}[Proof of Theorem \ref{lem:asrf}]
As with the quenched case, convexity on $[-1,0)$ and $(0,1]$ separately follows from the convexity of $\mathbb{J}_\eta(t)$ and $\mathbb{J}_{\eta^{\text{Inv}}}(t)$. 
Since Lemma \ref{lem:arflt} shows that $\mathbb{J}_\eta(t)$ is the Legendre transform of $\mathbf{\L}_\eta(\l)$ and since $\mathbf{\L}_\eta(\l) < \infty$ if and only if $\l \leq \lcrit(\eta)$, it again follows that 
\[
 \lim_{t\ra\infty} \frac{\mathbb{J}_\eta(t)}{t} = \lcrit(\eta).
\]
Since the same is true with $\eta$ replaced by $\eta^{\text{Inv}}$ and since $\lcrit(\eta)= \lcrit(\eta^{\text{Inv}})$, this proves that $\mathbb{I}_\eta(x)$ continuous at $x=0$. 
The continuity at $x=0$ also allows us to extend the convexity of $\mathbb{I}_\eta(x)$ to the closed intervals $[-1,0]$ and $[0,1]$ separately. 
Finally, the fact that $\mathbb{I}_\eta(x)$ and $I_\eta(x)$ have the same zero sets follows from the corresponding property for $\mathbb{J}_\eta(t)$ and $J_\eta(t)$ in Lemma \ref{arfpropTn}. 
\end{proof}

In addition to the properties of $I_\eta$ and $\mathbb{I}_\eta$ that we proved above, the crucial tool in transferring the large deviations from $T_n/n$ to $X_n/n$ will be the following lemma which gives a uniform quenched large deviations upper bound for slowdowns. 
\begin{lem}\label{I0ub}
Assume the distribution $\eta$ satisfies Assumptions \ref{asmerg} and \ref{easm} and that the RWRE is recurrent or transient to the right.
Then,
\begin{equation}\label{I0ublim}
 \limsup_{n\ra\infty} \frac{1}{n} \log \left\{ \max_{i} \sup_{\w \in \Omega_\eta} P_\w^{(0,i)}\left( \inf_{m\geq n} X_m \leq 0 \right) \right\} \leq - I_\eta(0),
\end{equation}
where $\Omega_\eta \subset \Omega$ is the support of the distribution $\eta$.
\end{lem}
\begin{rem}
A similar statement for RWRE on $\Z$ with holding times was proved in \cite[Lemma 4]{dgzLDPH}. The main difference in Lemma \ref{I0ub} is that the need to take the maximum over the initial starting height $i$ as well. 
The proof below is essentially an adaptation of the proof in \cite{dgzLDPH} but we present it here for completeness and to complete some minor gaps in the proof from \cite{dgzLDPH}. 
\end{rem}

\begin{proof}
For notational convenience, let $\s_n = \inf \{ m \geq n: X_m \leq 0 \}$ and
\[
 \b_n(\w) = \max_i P_\w^{(0,i)}( \s_n < \infty ), \quad\text{and}\quad \a_n = \sup_{\w \in \Omega_\eta} \b_n(\w).
\]
With this notation, the statement of the lemma is that $\limsup_{n\ra\infty} \frac{1}{n} \log \a_n \leq -I_\eta(0)$.
We will first show that 
\begin{equation}\label{threelimits}
 \lim_{n\ra\infty} \frac{1}{n} \log \a_n = \lim_{n\ra\infty} \frac{1}{n} \log \b_n(\w) = \lim_{n\ra\infty} \frac{1}{n} \log P_\w^{(0,i)}(\s_n < \infty), \quad \forall i \in [d], \quad \eta\text{ - a.s.}
\end{equation}
The proof will then be finished by deriving quenched large deviation estimates for $P_\w^{(0,i)}(\s_n < \infty)$. 

First of all, note that the uniform ellipticity in Assumption \ref{easm} implies that 
\[
 P_\w^{(0,i)} ( \s_n < \infty ) \geq \k^2 P_\w^{(0,j)}( \s_{n-2} < \infty ) \geq \k^2 P_\w^{(0,j)}( \s_n < \infty ), \quad \forall i,j \in [d]. 
\]
From this it is easy to see that $\b_{n+k}(\w) \geq \k^2 \b_n(\w) \b_k(\w)$ for all $n,k \geq 1$. 
This super-multipicative property, combined with the fact that $\b_n(\w) \geq \k^n$ for all $n\geq 1$, is enough to show that the limit  
\[
 B(\w) := \lim_{n\ra\infty} \frac{1}{n} \log \b_n(\w)
\]
exists, and $B(\w) \in [\log \k, 0]$. 
Moreover, since $\b_n(\w) \geq P_\w^{(0,i)}(\s_n < \infty) \geq \k^2 \b_n(\w)$ for any $i \in [d]$, the third limit in \eqref{threelimits} is also equal to $B(\w)$. 
To prove the first equality in \eqref{threelimits}, first note that it's obvious that $B(\w)\leq \liminf_{n\ra\infty} n^{-1} \log \a_n$. 
Let $A := \limsup_{n\ra\infty} n^{-1} \log \a_n$. 
Then for any fixed $k \geq 1$, it can be shown that there exists a $n_k \geq k$, $\w^{(k)} \in \Omega_\eta$ and $\ell_k \geq 1$ such that 
\[
 \frac{1}{n_k} \log \left\{ \max_i P_{\w^{(k)}}^{(0,i)} (\s_{n_k} < \ell_k ) \right\} > A - \frac{1}{k}. 
\]
Since for fixed $n,\ell<\infty$ the mapping $\w \mapsto \max_i P_\w^{(0,i)}( \s_n < \ell )$ is continuous on $\Omega_\eta$, we conclude that for every $k\geq 1$ there exists a relatively open subset $G_k \subset \Omega_\eta$ such that $\eta(\w \in G_k) > 0$ and
\[
 \frac{1}{n_k} \log \left\{ \max_i \sup_{\w \in G_k} P_{\w}^{(0,i)}( \s_{n_k} < \ell_k ) \right\} > A - \frac{1}{k}.
\]
Since $\eta(\w \in G_k) > 0$ and $\eta$ is an ergodic measure on environments, it follows that for any $k$ and $\eta$-a.e.\ environment $\w$ there exists a $d_k = d_k(\w) \leq 0$ such that $\theta^{-d_k} \w \in G_k$.
Then,
\begin{align*}
 \b_n(\w) 
&\geq \k^{d_k} P_{\th^{-d_k}\w}^{(0,j)}(\s_n < \infty) \\
&\geq \k^{d_k} \left\{ \k^2 \max_j P_{\th^{-d_k} \w }^{(0,j)}(\s_{n_k} < \ell_k ) \right\}^{\cl{n/n_k}}
\geq \k^{d_k} \left\{ \k^2 e^{n_k(A-\frac{1}{k})} \right\}^{\cl{n/n_k}}.
\end{align*}
Therefore, we can conclude for $\eta$-a.e.\ environment and any fixed $k$ that
\[
 B(\w) = \lim_{n\ra\infty} \frac{1}{n}\log \b_n(\w) \geq  \frac{2 \log \k}{n_k} + A - \frac{1}{k} .
\]
Taking $k\ra\infty$ we conclude that $B(\w) \geq A = \limsup_{n\ra\infty} n^{-1} \log \a_n$.

We have thus shown that the three limits in \eqref{threelimits} all equal the constant $A$. It remains to show that $A \leq -I_\eta(0)$. 
This is obviously true when $\lcrit(\eta) = I_\eta(0) = 0$, and so we need only to consider the case when $\lcrit(\eta)>0$. 
Recall that $t_0(\eta) = \L_\eta'(\l) < \infty$ when $\lcrit(\eta)>0$. Since $\lcrit(\eta) = \lcrit(\eta^{\text{Inv}}) > 0$ this implies that $t_0(\eta), t_0(\eta^{\text{Inv}}) < \infty$ and that $J_{\eta}(t)$ and $J_{\eta^{\text{Inv}}}(t)$ are non-decreasing on $[t_0(\eta),\infty)$ and $[t_0(\eta^{\text{Inv}}),\infty)$, respectively.
For convenience of notation let $P_\w^{(k,*)}(\cdot) = \max_j P_\w^{(k,j)}(\cdot)$. Then, for any $i \in [d]$, $K<\infty$ and 
$0<u < \frac{1}{t_0(\eta)} \wedge \frac{1}{t_0(\eta^{\text{Inv}})}$,
\begin{align}
& P_\w^{(0,i)}( \s_n < \infty) \nonumber \\
&\leq  P_\w^{(0,i)} \left( T_{\cl{nu}} \geq n \right) + P_\w^{(\cl{un},*)}(T_0 \in[n,\infty) ) \label{bsplit} \\
&\quad + \sum_{uK < k,\ell \leq K } P_\w^{(0,i)}\left( T_{\cl{nu}} \in \left[ \tfrac{(k-1)n}{K}, \tfrac{kn}{K} \right) \right) P_\w^{(\cl{nu},*)}\left( T_0 \in \left[ \tfrac{(\ell-1)n}{K}, \tfrac{\ell n}{K} \right) \right) \b_{\cl{n(1-\frac{k+\ell}{K}) }}(\w), \nonumber
\end{align}
where by convention we let $\b_{-m}(\w) = 1$ for any $m\geq 0$. 
Note that in the last sum above we have restricted $k,\ell > uK$ since otherwise the probabilities inside the sum are zero. 

Next we derive large deviation upper bounds for all of the terms on the right side of \eqref{bsplit}. 
For the first term on the right, Theorem \ref{QLDPTn} implies that 
\[
 \limsup_{n\ra\infty} \frac{1}{n} P_\w^{(0,i)} \left( T_{\cl{nu}} \geq n \right) \leq - u J_\eta\left( \frac{1}{u} \right) = - I_\eta\left( u \right),
\]
where in the first inequality we used that $1/u > t_0(\eta)$. 
Similarly, we claim that 
\begin{equation}\label{shiftqldp}
  \limsup_{n\ra\infty} \frac{1}{n} P_\w^{(\cl{nu},*)} \left( T_0 \in [n,\infty) \right) \leq - u J_{\eta^{\text{Inv}}}\left( \frac{1}{u} \right) = - I_\eta\left( -u \right). 
\end{equation}
Note that \eqref{shiftqldp} does not follow directly from Theorem \ref{QLDPTn} since the starting location of the random walk is changing with $n$. However, it can be shown that the proof of Theorem \ref{QLDPTn} still carries through in this case. Indeed, the key to the large deviation upper bound is the computation of the asymptotics of the quenched log moment generating function, and as was shown above in \eqref{reverseqlmgflim}, $\lim_{n\ra\infty} n^{-1} \log E_\w^{(n,j)}[ e^{\l T_0} \ind{T_0 < \infty} ] = \L_{\eta^{\text{Inv}}}(\l).$
From this the proof of \eqref{shiftqldp} is standard. 
Finally, in the same way it can be shown that for fixed $k,\ell$
\begin{align}
& \limsup_{n\ra\infty} \frac{1}{n} \log \left\{ P_\w^{(0,i)}\left( T_{\cl{nu}} \in \left[ \tfrac{(k-1)n}{K}, \tfrac{kn}{K} \right) \right) P_\w^{(\cl{nu},*)}\left( T_0 \in \left[ \tfrac{(\ell-1)n}{K}, \tfrac{\ell n}{K} \right) \right) \b_{\cl{n(1-\frac{k+\ell}{K}) }}(\w) \right\} \nonumber \\
&\leq - \left\{ \inf_{t \in [\frac{k-1}{Ku},\frac{k}{Ku}] } u J_{\eta}(t) + \inf_{t \in [\frac{\ell-1}{Ku},\frac{\ell}{Ku}] } u J_{\eta^{\text{Inv}}}(t) \right\} + \left( 1- \frac{k+\ell}{K} \right)_+ A \nonumber \\
&\leq - \left\{ u J_{\eta}\left( \frac{k}{Ku} \right) + u J_{\eta^{\text{Inv}}}\left( \frac{\ell}{Ku} \right) \right\} + \left( 1- \frac{k+\ell}{K} \right)_+ A  + u \Delta_{u,K}, \label{DuKerror}
\end{align}
where the error term $\Delta_{u,K}$ vanishes as $K\ra\infty$ for any fixed $u$ (this follows from the fact that $J_\eta$ and $J_{\eta^\text{Inv}}$ are uniformly continuous on $[1,1/u]$). 
For the terms inside the braces in \eqref{DuKerror} we have 
\[
 u J_{\eta}\left( \frac{k}{Ku} \right) + u J_{\eta^{\text{Inv}}}\left( \frac{\ell}{Ku} \right) = \frac{k}{K} I_\eta\left( \frac{Ku}{k} \right) + \frac{\ell}{K} I_\eta\left(\frac{-Ku}{\ell} \right) \geq \frac{k+\ell}{K} I_\eta(0),
\]
where the last inequality follows from the convexity of $I_\eta$. 

Since we are trying to show that $A\leq -I_\eta(0)$, we may assume for contradiction that $A + I_\eta(0) > 0$ in which case 
\begin{align*}
 \eqref{DuKerror} 
&\leq -  \frac{k+\ell}{K} I_\eta(0) + \left( 1- \frac{k+\ell}{K} \right)_+ A  + u \Delta_{u,K} \\
&\leq u \Delta_{u,K} +
\begin{cases}
 -I_\eta(0) & \text{if } k+\ell \geq K \\
 A - 2 u (I_\eta(0)+A) & \text{if } k,\ell > uK, \text{ and } k+\ell < K. 
\end{cases} 
\end{align*}
Combining all of the above large deviation estimates for the terms on the right side of \eqref{bsplit} 
and using the fact that $A = \lim_{n\ra\infty} n^{-1} \log P_\w^{(0,i)}( \s_n < \infty)$, we obtain that 
\[
A 
\leq \max\left\{ -I_\eta\left(-u \right), \, -I_\eta\left(u\right), \, -I_\eta(0) , \, A-2u(I_\eta(0)+A) \right\}  + u\Delta_{u,K}.
\]
Letting $K\ra\infty$ we get the same inequality without the last term since $\Delta_{u,K} \ra 0$ for $u$ fixed. 
Since we assumed for contradiction that $I_\eta(0) + A > 0$, the last term in the maximum is strictly less that $A$, and thus the maximum must be attained by one of the first three terms. Then taking $u\ra 0$
we conclude that $A \leq -I_\eta(0)$, 
contradicting our previous assumption that $A > -I_\eta(0)$. 
\end{proof}

\subsection{Proof of Theorem \ref{th:qldpXn}}

We are now ready to prove the quenched large deviation principle for $X_n/n$ as stated in Theorem \ref{th:qldpXn}. 
Note that by symmetry, we may assume that $\eta$ is such that the RWRE is recurrent or transient to the right so that $\vp \geq 0$. 
Since $I_\eta(x)$ is non-increasing on $[-1,\vp]$ and non-decreasing on $[\vp,1]$, to prove the large deviation upper bound it is enough to show that 
\begin{equation} \label{qubs1}
  \limsup_{n\ra\infty} \frac{1}{n} \log P_\w^\pi(X_n \geq xn) \leq - I_\eta(x), \quad \eta\text{ - a.s.} \quad \forall x \geq \vp, 
\end{equation}
and 
\begin{equation}\label{qubs2}
  \limsup_{n\ra\infty} \frac{1}{n} \log P_\w^\pi(X_n \leq xn) \leq - I_\eta(x), \quad \eta\text{ - a.s.} \quad \forall x \leq \vp, 
\end{equation}

To prove \eqref{qubs1}, note that Theorem \ref{QLDPTn} implies that  
\begin{align*}
 \limsup_{n\ra\infty} \frac{1}{n} \log P_\w^\pi(X_n \geq xn) 
&\leq \limsup_{n\ra\infty} \frac{1}{n} \log P_\w^\pi(T_{\cl{xn}} \leq n ) \nonumber \\
&\leq -x \inf_{t \leq 1/x} J_\eta(t) =  - x J_\eta(1/x) = -I_\eta(x), \quad \forall x \geq \vp, 
\end{align*}
where the second to last equality follows from the fact that $J_\eta$ is non-increasing on $(-\infty,t_0]$ and $t_0 = 1/\vp$ (see Lemmas \ref{Jetaprop} and \ref{lem:t0}).

To prove corresponding large deviation upper bounds for the left tails note that for any $x>0$ and a fixed $K<\infty$, by decomposing according to the hitting time $T_{\fl{xn}}$ we obtain
\begin{equation}\label{Xnqlefttail}
\begin{split}
 P_\w^{\pi}( X_n \leq xn )
&\leq P_\w^\pi( T_{\fl{xn}} \geq n ) \\
&\quad + \sum_{k=1}^K P_\w^\pi\left( T_{\fl{xn}} \in \left[ \frac{(k-1)n}{K}, \frac{kn}{K} \right)  \right)
\left\{ \max_i P_{\th^{\fl{nx}}\w}^{(0,i)} \left( \inf_{t \geq n(1-k/K)} X_t \leq 0 \right) \right\}.
\end{split}
\end{equation}
Lemma \ref{I0ub}, together with Theorem \ref{QLDPTn}, then implies that for $x \in (0,\vp]$,
\begin{equation}\label{qKdecomp1}
\begin{split}
 \limsup_{n\ra\infty} & \frac{1}{n} \log P_\w^{\pi}( X_n \leq xn ) \\
&\leq - \min\left\{ x \inf_{t \geq 1/x} J_\eta(t), \,  \min_{k\leq K}  \left\{ \inf_{t \in \left[\frac{k-1}{Kx}, \frac{k}{Kx}\right]} x J_\eta(t) + \left(1-\frac{k}{K}\right) I_\eta(0) \right\} \right\} \\
&=- \min\left\{ x J_\eta(1/x) ,  \, \min_{k\leq K}  \left\{ \inf_{s \in \left[\frac{k-1}{K}, \frac{k}{K}\right]} s I_\eta(x/s) + \left(1-s \right) I_\eta(0) - \left(\frac{k}{K} - s\right) I_\eta(0)\right\} \right\} \\
&\leq - \min\left\{ I_\eta(x) ,  \, \min_{k\leq K}  \left\{ \inf_{s \in \left[\frac{k-1}{K}, \frac{k}{K}\right]} I_\eta(x) - \left(\frac{k}{K} - s\right) I_\eta(0)\right\} \right\} \\
&= - I_\eta(x) + \frac{1}{K} I_\eta(0),
\end{split}
\end{equation}
where in the first equality we used that $\inf_{t \geq 1/x} J_\eta(t) = J_\eta(1/x)$ since $J_\eta(t)$ is non-decreasing on $[1/\vp,\infty)$, and in the second to last line we use the fact that $I_\eta(x)$ is convex in $x$. 
Finally, letting $K \ra\infty$ proves \eqref{qubs2} when $x\in(0,\vp]$. 
Similarly, if $x < 0$ then $\{X_n \leq xn\} \subset \{ T_{\fl{xn}} \leq n \}$ and by decomposing according to the hitting time $T_{\fl{xn}}$ we obtain
\begin{align*}
 P_\w^{\pi}( X_n \leq xn )
&\leq \sum_{k=1}^K P_\w^\pi\left( T_{\fl{xn}} \in \left( \frac{(k-1)n}{K}, \frac{kn}{K} \right]  \right)
\left\{ \max_i P_{\th^{\fl{nx}}\w}^{(0,i)} \left( \inf_{t \geq n(1-k/K)} X_t \leq 0 \right) \right\}. 
\end{align*}
From this, the quenched large deviation principle for $T_{-n}/n$ together with Lemma \ref{I0ub} implies that 
\begin{equation}\label{qKdecomp2}
 \begin{split}
 \limsup_{n\ra\infty} \frac{1}{n} \log P_\w^{\pi}( X_n \leq xn ) 
&\leq - \min_{k\leq K}  \left\{ \inf_{t \in \left[\frac{k-1}{K|x|}, \frac{k}{K|x|}\right]} |x| J_{\eta^{\text{Inv}}}(t) + \left(1-\frac{k}{K}\right) I_\eta(0) \right\} \\
&= - \min_{k\leq K}  \left\{ \inf_{s \in \left[\frac{k-1}{K}, \frac{k}{K}\right]} s I_\eta(x/s) + (1-s)I_\eta(0)-\left(\frac{k}{K}-s\right) I_\eta(0) \right\} \\
&\leq - \min_{k\leq K}  \left\{ \inf_{s \in \left[\frac{k-1}{K}, \frac{k}{K}\right]} I_\eta(x) - \left(\frac{k}{K} - s\right) I_\eta(0)\right\}  \\
&= - I_\eta(x) + \frac{1}{K} I_\eta(0). 
\end{split}
\end{equation}
Again, letting $K \ra 0$ proves \eqref{qubs2} for $x < 0$. Finally, since $ P_\w^\pi(X_n \leq 0) \leq P_\w^\pi( \inf_{m\geq n} X_m \leq 0 )$, Lemma \ref{I0ub} implies that \eqref{qubs2} holds for $x=0$ as well. 
This completes the proof of the large deviation upper bound in Theorem \ref{th:qldpXn}. 

For the large deviation lower bound, it is enough to show that
\begin{equation}\label{qlbs}
 \lim_{\d\ra 0} \liminf_{n\ra\infty} \frac{1}{n} \log P_\w^\pi( |X_n - xn| < \d n ) \geq -I_\eta(x), \quad \eta\text{ - a.s.,} \quad \forall x\in\R. 
\end{equation}
To show this when $x\neq 0$, since 
$|X_k-X_{k-1}| \leq 1$ for all $k\geq 1$ it follows that 
\[
 P_\w^\pi( |X_n - xn| < \d n ) \geq P_\w^\pi\left( |T_{\fl{xn}} - n | < \d n - 1 \right). 
\]
Then \eqref{qlbs} follows easily from the quenched large deviation principle for $T_n/n$ when $x>0$ or from the quenched large deviation principle for $T_{-n}/n$ when $x<0$. 
To show \eqref{qlbs} for $x=0$, note that the continuity of $I_\eta$ implies that for any $\e>0$ there exists a $\d_0 = \d_0(\e)> 0$ such that $I_\eta(\d/2) < I_\eta(0) + \e$ for all $\d \in (0,\d_0)$. Applying \eqref{qlbs} with $x=\d/2$ implies that there exists a $\d'< \d/2$ such that 
\begin{align*}
\liminf_{n\ra\infty} \frac{1}{n} \log P_\w^\pi( |X_n| < \d n ) 
&\geq \liminf_{n\ra\infty} \frac{1}{n} \log P_\w^\pi( |X_n - \d n/2| < \d' n ) \\
&\geq -I_\eta(\d/2) - \e > -I_\eta(0) - 2\e. 
\end{align*}
Since $\e>0$ was arbitrary, this proves \eqref{qlbs} for $x=0$ and thus finishes the proof of the quenched large deviations lower bound for $X_n/n$. 

\subsection{Proof of Theorem \ref{th:aldpXn}}
To prove the averaged large deviations lower bound for $X_n/n$ it is enough to show that 
\begin{equation}\label{alb}
 \lim_{\d\ra 0} \liminf_{n\ra\infty} \frac{1}{n} \log \P^\pi( |X_n - nx| < n\d ) \geq -\mathbb{I}_\eta(x), \quad \forall x \in (-1,1),
\end{equation}
As in the quenched case, when $x\neq 0$ this follows from Theorem \ref{th:aldpTn} and the fact that 
\[
 \P^\pi( |X_n - xn| < \d n ) \geq  \P^\pi\left( |T_{\fl{xn}} - n | < \d n - 1 \right), \quad \forall x\neq 0, \, \d>0. 
\]
The same argument as in the quenched case then shows that \eqref{alb} can be extended to $x=0$ by the fact that $\mathbb{I}_\eta(x)$ is continuous at $x=0$.

To prove the matching large deviation upper bound we will show below that it is enough to prove that
\begin{equation}\label{aubrt}
 \limsup_{n\ra\infty} \frac{1}{n} \log \P^\pi( X_n \geq xn ) \leq - \mathbb{I}_\eta(x), \quad \forall x \geq \vp,
\end{equation}
and 
\begin{equation}\label{aublt}
 \limsup_{n\ra\infty} \frac{1}{n} \log \P^\pi( X_n \leq xn ) \leq - \mathbb{I}_\eta(x), \quad \forall x \leq \vp. 
\end{equation}
As with the quenched large deviation principle for $X_n/n$ we will assume without loss of generality that the RWRE is recurrent or transient to the right. 
Then, the upper bound for right tails \eqref{aubrt} follows easily from Theorem \ref{th:aldpTn} since 
\begin{align*}
\limsup_{n\ra\infty} \frac{1}{n} \log \P^\pi( X_n \geq xn) 
&\leq \limsup_{n\ra\infty} \frac{1}{n} \log \P^\pi( T_{\cl{xn}} \leq n) \\
&\leq -x \inf_{t\leq1/x} \mathbb{J}_\eta(t)
= -x \mathbb{J}_\eta(1/x) = -\mathbb{I}_\eta(x), 
\end{align*}
where we used the fact that $\mathbb{J}_{\eta}(t)$ is non-increasing on $(-\infty,t_0]$ and $t_0 = 1/ \vp$.

To prove \eqref{aublt} for $x \in (0,\vp]$, taking expectations of \eqref{Xnqlefttail} implies that for any fixed $K\geq 1$
\begin{align*}
&\P^{\pi}( X_n \leq xn ) \\
&\leq \P^\pi( T_{\fl{xn}} \geq n ) 
+ \sum_{k=1}^K \P^\pi\left( T_{\fl{xn}} \in \left[ \tfrac{(k-1)n}{K}, \tfrac{kn}{K} \right)  \right)
\left\{ \sup_{\w \in \Omega_\eta} \max_i P_{\w}^{(0,i)} \left( \inf_{t \geq n(1-k/K)} X_t \leq 0 \right) \right\}.
\end{align*}
Then, applying Theorem \ref{th:aldpTn} and Lemma \ref{I0ub} (and the fact that $I_\eta(0) = \mathbb{I}_\eta(0) = \lcrit(\eta)$) and repeating the steps in \eqref{qKdecomp1} we obtain that for any $x \in (0,\vp]$,
\begin{align*}
& \limsup_{n\ra\infty} \frac{1}{n} \log \P^{\pi}( X_n \leq xn )\\
&\quad \leq - \min\left\{ x \inf_{t \geq 1/x} \mathbb{J}_\eta(t), \, \min_{k\leq K} \left\{  \inf_{t \in [ \frac{k-1}{Kx}, \frac{k}{Kx} ] } x \mathbb{J}_\eta(t) + (1-\tfrac{k}{K})\mathbb{I}_\eta(0) \right\} \right\} \\
&\quad= - \mathbb{I}_\eta(x) + \frac{1}{K} \mathbb{I}_\eta(0). 
\end{align*}
Then, taking $K\ra \infty$ proves \eqref{aublt} when $x>0$. 
The proof of \eqref{aublt} when $x<0$ is similar, mimicing the steps in \eqref{qKdecomp2} and using the averaged large deviation principle for $T_{-n}/n$ instead. Finally, \eqref{aublt} holds when $x=0$ by Lemma \ref{I0ub}. 

We still need to show that indeed \eqref{aubrt} and \eqref{aublt} imply the general large deviations upper bound
\[
 \limsup_{n\ra\infty} \frac{1}{n} \log \P^\pi( X_n/n \in F ) \leq - \inf_{x \in F} \mathbb{I}_\eta(x), \quad \text{ for all closed } F.
\]
In order for this to be true, it is necessary that the averaged rate function $\mathbb{I}_\eta(x)$ is non-increasing on $[-1,\vp]$ and non-decreasing $[\vp,1]$.  
If $\lcrit(\eta) = 0$ then this is obvious since $\mathbb{I}_\eta(x)$ is convex, non-negative, and $\mathbb{I}_\eta(\vp) = 0$. 
On the other hand, if $\lcrit(\eta) > 0$ then we need a different argument since we don't have a direct proof that $\mathbb{I}_\eta(x)$ is convex. 
Since $\mathbb{I}_\eta(0) = \lcrit(\eta) > 0$ it follows that $\vp \neq 0$ and so we may assume without loss of generality that $\vp > 0$. Since $\mathbb{I}_\eta(x)$ is non-negative and convex on $[0,1]$, $\mathbb{I}_\eta(x)$ is non-decreasing on $[\vp,1]$ and non-increasing on $[0,\vp]$. It remains only to show that $\mathbb{I}_\eta(x)$ is non-increasing on $[-1,0]$. To this end, fix $x<y\leq 0$. Then, \eqref{aublt} and \eqref{alb} imply that 
\begin{align*}
 -\mathbb{I}_\eta(x) \leq \lim_{\d\ra 0} \liminf_{n\ra\infty} \frac{1}{n} \log \P^\pi( |X_n - nx| < n\d) 
&\leq \limsup_{n\ra\infty} \frac{1}{n} \log \P^\pi( X_n \leq y n ) \leq -\mathbb{I}_\eta(y), 
\end{align*}
and so $\mathbb{I}_\eta(x)$ is indeed non-increasing on $[-1,0]$. 

We close the discussion of the averaged large deviation principle for $X_n/n$ by noting that the variational formula for $\mathbb{J}_\eta(t)$ in \eqref{eq:arfTn} implies a corresponding variational formula $\mathbb{I}_\eta(x)$. Indeed, we claim that 
\begin{equation}\label{eq:varformI}
 \mathbb{I}_\eta(x) = \inf_{\a \in M_1^e(\Omega_\k) \cap \mathcal{M}_\eta} \{ I_\a(x) + |x| h(\a|\eta) \}. 
\end{equation}
Recall from Lemma \ref{lem:hfinite} that $h(\a|\eta) = \infty$ for $\a \notin \mathcal{M}_\eta$, and thus for $x\neq 0$ the infimum in \eqref{eq:varformI} can be extended to $\a \in M_1^e(\Omega_\k)$. 
Then, \eqref{eq:varformI} follows easily from \eqref{eq:arfTn} and the formula for $\mathbb{I}_\eta(x)$ when $x\neq 0$. To show \eqref{eq:varformI} when $x=0$ note that 
\[
 \inf_{\a \in M_1^e(\Omega_\k) \cap \mathcal{M}_\eta} I_\a(0) = \inf_{ \a \in M_1^e(\Omega_\k) \cap \mathcal{M}_\eta} \lcrit(\a) = \lcrit(\eta), 
\]
where the last equality follows from Lemma \ref{Laldef}.

\appendix

\section{RWRE with bounded step sizes}\label{sec:boundedjumps}

It is well known that RWRE on a strip can be thought of as a generalization of RWRE on $\Z$ with bounded jumps. Indeed, if a RWRE on $\Z$ has jump sizes of at most $d\geq 2$, then by identifying points $(k,i) \in \Z \times [d]$ with the point $x=kd +i-1 \in \Z$ we can interpret the random walk as occuring on the strip. 
However, not all natural RWRE on $\Z$ with bounded jumps satisfy the uniform ellipticity in Assumption \ref{easm} when thought of as RWRE on a strip. 
In particular, there have been several results on what we will call $(L,R)$-RWRE \cite{kRWREBJ,bL1RWRE,bRWRELyapunov,hwL1Branching,hz1RBranching}; that is, RWRE on $\Z$ with jumps of at most $L$ steps to the left and at most $R$ steps to the right. We will consider $(L,R)$-RWRE that satisfy the following uniform ellipticity assumption
\begin{equation}\label{LRue}
\eta\left( P_\w( X_1 \in [-L,R] ) = 1 \right) = 1, \quad \text{and}\quad \eta\left( P_\w( X_1 = z ) \geq \k, \, \forall z \in [-L,R] \backslash \{0\} \right) = 1. 
\end{equation}
(Note that the second requirement in \eqref{LRue} allows, but does not require, the possibility that the RWRE may stay at its current location with positive probability.)
Such random walks can be viewed as a random walk on the strip $\Z \times [d]$ with $d = \max\{ L, R \}$. 
If $L=R=d$, then it is easy to see that Assumption \ref{easm} is satisfied. 
On the other hand, if $L\neq R$ then Assumption \ref{easm} is not satisfied. 
For instance, if $L>R$ when we translate the model to the strip $\Z \times [L]$ we have that 
\[
 \sum_j p_k(i,j) = 0, \, \forall i \in [1,L-R], 
 \quad \text{and}\quad 
 \sum_i p_k(i,j) = 0, \, \forall j \in [R+1, L]. 
\]
Thus both \eqref{onestepellipticity} and \eqref{entryheightellipticity} are violated for such RWRE. 

The most crucial way in which we used Assumption \ref{easm} was in the proofs of the existence of the vectors $\mu_n(\l)$ and $\nu_n(\l)$ where we used that the matrices $\Phi_k(\l)$ have entries bounded uniformly below. 
For RWRE on the strip $\Z \times [L]$ coming for $(L,R)$-RWRE on $\Z$ with $L>R$ it follows that $\P^{(0,i)}( T_1 < \infty, \, Y_{T_1} = j) = 0$ for all $j \in [R+1,L]$. Thus, $\Phi_k(\l)(i,j) = 0$ if $j \in [R+1,L]$. On the other hand, all other entries can be uniformly bounded below away from zero. In fact the ellipticity assumptions imply that $\Phi_k(\l)(i,j) \geq \k^L e^{\l L}$ for all $j \in [1,R]$. We will show how these facts can be used to prove Lemmas \ref{mudeflem} and \ref{nudeflem} for such RWRE on the strip. 

For convenience of notation, we will write the matrices $\Phi_k(\l)$ in block matrix notation as 
\[
 \Phi_k(\l) = 
\left(
 \begin{array}{c|c}
 A_k(\l) & \mathbf{0} \\
 \hline 
 B_k(\l) & \mathbf{0} 
\end{array}
\right)
\]
where $A_k(\l)$ is an $R\times R$ matrix and $B_k(\l)$ is $(L-R) \times R$ matrix. As noted above, the entries of $A_k(\l)$ and $B_k(\l)$ can be uniformly bounded away from 0. Also, a similar argument as in the proof of Lemma \ref{PhiUboundlem} shows that the entries can be uniformly bounded above for each $\l \leq \lcrit(\eta)$. That is, there exists some $c_\l > 0$ such that 
\begin{equation}\label{AUbound}
 c_\l \leq A_k(\l)(i,j), B_k(\l)(i,j) \leq \frac{1}{c_\l}. 
\end{equation}
If we adopt the notation $A_{[k,n]}(\l) = A_k(\l)A_{k+1}(\l) \cdots A_n(\l)$ for $k\leq n$ then it is clear that 
\begin{equation}\label{blockproduct}
 \Phi_{[k,n]}(\l) = 
 \left(
 \begin{array}{c|c}
 A_{[k,n]}(\l) & \mathbf{0} \\
 \hline 
 B_k(\l)A_{[k-1,n]} & \mathbf{0} 
\end{array}
\right)
\end{equation}
By the uniform bounds \eqref{AUbound} on the entries of $A_k(\l)$ we can conclude from Lemma \ref{pr} that there exists a vector $\tilde{\mu}_n(\l) \in \R^{R}$ with non-negative entries summing to 1 such that 
\begin{equation}\label{tmuerror}
 \sup_{\mathbf{0} \neq \pi \geq 0} \left\| \frac{\pi A_{[m,n-1]}(\l) }{ \pi A_{[m,n-1]}(\l) \one } - \tilde\mu_n(\l) \right\|_1 \leq \frac{2 (1-c_\l^4)^{n-m-1}}{c_\l^4}, \quad \forall m < n. 
\end{equation}

Now, let $\mu_n(\l) = (\tilde{\mu}_n(\l), \mathbf{0}) \in \R^L$ (i.e., append the vector $\tilde\mu_n(\l)$ with $L-R$ zeros at the end). 
Then, it follows from the block matrix representation \eqref{blockproduct} that
\begin{align*}
 \left\| \frac{e_i \Phi_{[m,n-1]}(\l) }{e_i \Phi_{[m,n-1]}(\l) \one} - \mu_n(\l) \right\|_1
 &\leq
 \begin{cases}
  \left\| \frac{e_i A_{[m,n-1]}(\l) }{e_i A_{[m,n-1]}(\l) \one} - \tilde\mu_n(\l) \right\|_1 & \text{if } i \in [1,R] \\
  \left\| \frac{e_i B_m A_{[m+1,n-1]}(\l) }{e_i B_m A_{[m+1,n-1]}(\l) \one} - \tilde\mu_n(\l) \right\|_1 & \text{if } i \in [R+1,L]
 \end{cases} \\
 &\leq \frac{2 (1-c_\l^4)^{n-m-2}}{c_\l^4},
\end{align*}
where the last inequality follows from \eqref{tmuerror} with $\pi=e_i$ when $i \in [1,R]$ and with $\pi = e_i B_m(\l)$ when $i \in [R+1,L]$. 

The proof of the existence of $\nu_n(\l)$ with corresponding error bounds is similar but slightly more involved. 
First of all, note that since the entries of $A_k(\l)$ are uniformly bounded above and below, for every $k\in\Z$ there exists a $\s_k(\l) \in \R^R$ such that 
\begin{equation}\label{skerror}
 \left\| \s_{k,n}(\l) - \s_k(\l) \right\|_1 \leq \frac{2}{c_\l^4}(1-c_\l^4)^{n-k}, \quad \text{ where } \s_{k,n}(\l) = \frac{ A_{[k,n]}(\l) \one}{ \one^t A_{[k,n]}(\l) \one } \quad \text{ for } k \leq n. 
\end{equation}
Let $\tilde{\s}_{k,n}(\l)$ and $\tilde\s_k(\l)$ denote vectors in $\R^L$ formed by adding $L-R$ zeros to the end of $\s_{k,n}(\l)$ and $\s_k(\l)$, respectively. Then, using the block matrix representation in \eqref{blockproduct} it can be shown that 
\[
 \frac{ \Phi_{[k,n]}(\l) \one }{ \one^t \Phi_{[k,n]}(\l) \one } = 
\frac{ \Phi_k(\l) \tilde\s_{k+1,n}(\l) }{ \one^t \Phi_k(\l) \tilde\s_{k+1,n}(\l) }  
\]
Since \eqref{skerror} implies that $\tilde\s_{k+1,n} \ra \tilde\s_{k+1}$ as $n\ra\infty$, it follows that
\[
 \lim_{n\ra\infty}  \frac{ \Phi_{[k,n]}(\l) \one }{ \one^t \Phi_{[k,n]}(\l) \one } 
= \frac{ \Phi_k(\l) \tilde\s_{k+1}(\l) }{ \one^t \Phi_k(\l) \tilde\s_{k+1}(\l) }
=: \nu_k(\l). 
\]
Finally, the error bounds in \eqref{skerror} and the uniform bounds on the non-zero entries of $\Phi_k(\l)$ can be used to show that for any $k\leq n$
\[
 \left\| \frac{ \Phi_{[k,n]}(\l) \one }{ \one^t \Phi_{[k,n]}(\l) \one } - \nu_k(\l) \right\|_\infty \\
 = \left\| \frac{ \Phi_k(\l) \tilde\s_{k+1,n}(\l) }{ \one^t \Phi_k(\l) \tilde\s_{k+1,n}(\l) } - \frac{ \Phi_k(\l) \tilde\s_{k+1}(\l) }{ \one^t \Phi_k(\l) \tilde\s_{k+1}(\l) } \right\|_\infty \leq C (1-c_\l^4)^{n-k-1},
\]
where the constant $C$ depends only on $\l$. 

Having shown the existence of the vectors $\mu_k(\l)$ and $\nu_k(\l)$ as well as error bounds similar to \eqref{muerror} and \eqref{nuerror}, one can adapt the rest of the proofs of the quenched and averaged large deviation principles for $T_n/n$ and $X_n/n$ with a few minor technical modifications. 
The details are left to the interested reader. 

\bibliographystyle{alpha}
\bibliography{Strip}

\end{document}